\documentclass[11pt,a4paper]{amsart}%
\usepackage{hyperref}
\usepackage{pdfsync}
\usepackage{dsfont}
\usepackage{color}
\usepackage{amsthm}
\usepackage{enumerate}
\usepackage{amsfonts}
\usepackage{amssymb}%
\usepackage{mathtools}
\usepackage{braket}
\usepackage{bm}
\usepackage{amsmath}
\usepackage{graphicx,transparent}
\usepackage{caption}
\usepackage{subcaption}
\usepackage{fullpage}
\usepackage{pb-diagram}
\numberwithin{equation}{section}
\newtheorem{theorem}{Theorem}[section]

\newtheorem{example}[theorem]{Example}
\newtheorem{lemma}[theorem]{Lemma}

\newtheorem{proposition}[theorem]{Proposition}
\theoremstyle{remark}
\newtheorem{remark}[theorem]{Remark}

\newcommand{\1}{\mathds{1}}
\newcommand{\R}{\mathds{R}}
\newcommand{\N}{\mathds{N}}

\newcommand{\converges}[1]{ \overset{#1}{\longrightarrow}} 

\newcommand{\bnu}{\boldsymbol\nu}
\DeclareMathOperator*{\argmin}{arg\,min}
\newcommand{\M}{\mathcal{M}}

\newcommand{\x}{\mathbf{x}}
\newcommand{\y}{\mathbf{y}}

\newcommand{\veps}{\varepsilon}

\DeclareMathOperator{\Prob}{\mathbb{P}}

\DeclareMathOperator{\sign}{sign}

\DeclareMathOperator{\divergence}{div}

\newcommand{\blue}{\color{blue}}
\definecolor{mygreen}{rgb}{0.1,0.75,0.2}

\newcommand{\nc}{\normalcolor}

\title{A new analytical approach to consistency and overfitting in regularized empirical risk minimization}
\author{Nicol\'as Garc\'ia Trillos \and Ryan Murray} 

\newcommand{\Addresses}{{% additional braces for segregating \footnotesize
  \bigskip
  \footnotesize
  \textsc{Division of Applied Mathematics, Brown University,
    Providence, RI, 02912, USA. }\par\nopagebreak  \textit{Email:} \texttt{nicolas\_garcia\_trillos@brown.edu}
    
  \medskip
 \textsc{Mathematics Department, The Pennsylvania State University, University Park, PA 16802, USA.}\par\nopagebreak
  \textit{Email:} \texttt{rwm22@psu.edu}
}}

\begin{document}

% 	49J55  	Problems involving randomness [See also 93E20]
% 	49J45  	Methods involving semicontinuity and convergence; relaxation
% 	68R10  		Discrete mathematics in relation to computer science: Graph theory
% 	62G20  		Nonparametric inference: Asymptotic properties
%     60D05     Geometric probability and stochastic geometry

\newcounter{broj1}
\date{\today}
\begin{abstract}
This work considers the problem of binary classification: given training data $\x_1, \dots, \x_n$ from a certain population, together  with associated labels $\y_1,\dots, \y_n \in \left\{0,1 \right\}$, determine the best label for an element $\x$ not among the training data. More specifically, this work considers a variant of the regularized empirical risk functional which is defined intrinsically to the observed data and does not depend on the underlying population. Tools from modern analysis are used to obtain a concise proof of  asymptotic consistency as regularization parameters are taken to zero at rates related to the size of the sample. These analytical tools give a new framework for understanding overfitting and underfitting, and rigorously connect the notion of overfitting with a loss of compactness. 
\end{abstract}
\maketitle

\section{Introduction}

The problem of classification is one of the most important problems in machine learning and statistics. In this paper we consider the problem of binary classification: given training data $\x_1, \dots, \x_n$ from a population, together  
with associated labels $\y_1,\dots, \y_n \in \left\{0,1 \right\}$, determine the best label for an element $\x$ not among the training data. The $\x$ variables represent the values of certain features identifying individuals/objects in a given population; on the other hand, the $\y$ variables represent a group each individual belongs to. 
The classification problem is thus to construct, using the available training data $(\x_i,\y_i)_{i=1\dots n}$, a function, called a classifier, mapping features $\x$ to labels $u(\x)$, which reflects patterns or trends exhibited in the samples. In some sense, the goal can be posed as ``learning'' relevant aspects of the underlying geometry of the population by observing only a finite number of samples.

Here we follow the standard assumption that the data $\left\{ (\x_i,\y_i) \right\}_{i}$ are independent samples of some unknown ground-truth distribution $\bnu$. This means that $\y_i$ is not simply obtained by evaluating a function at $\x_i$, but instead $\y_i$ is randomly chosen from a distribution that depends on $\x_i$. In other words, the labels in the training data are randomly obtained from a distribution that depends on the feature values:
\[ \y_i \sim \Prob(\y_i =\cdot | \x= \x_i  ) .\]
For our purposes, this assumption gives a robust means to account for external sources of noise and for internal uncertainty associated to an object/individual (for example, the features may not always give all of the relevant information about an individual). It is also reasonable to assume that objects with similar features have similar labels, which in this probabilistic setting means that the distribution $\Prob(\y = \cdot \:  | \x=x )$ varies continuously in $x$.

By way of definition, a classifier is a function $u: D \rightarrow \{0,1 \}$, where we use $D$ to denote the space of features for the given population. The performance, or ``goodness'' of any classifier is measured in terms of some risk functional. The risk functional that we consider in this paper is the average misclassifications error for data sampled from the distribution $\bnu$. More precisely, given a classifier $u : D \rightarrow \{ 0,1\}$, we define its \textit{risk} as
\[  R(u):= \mathbb{E}( \lvert u(\x) - \y  \rvert)= \int_{D \times \{0,1 \}}|u(x) - y| d \bnu(x,y).  \]
With respect to this risk functional, the best classifier (i.e. the one that minimizes the risk) is the \textit{Bayes classifier}, which is the function  $u_B$ defined as
\[ u_B(x) :=  \begin{cases}1 & \text{ if } \Prob(\y=1 | \x= x  ) > 1/2{\blue ,} \\  0  & \text{ otherwise}. \end{cases} \]
A central difficulty in the classification problem is that $\bnu$ is unknown, and thus we can not compute either $R(u)$ or $u_B$. In fact, in some cases the extent of $D$, or in other words the support of $\bnu$, may be unknown. Given that the Bayes classifier is the best classifier, a reasonable goal is then to construct a classifier based completely on the training data, in such a way that it approximates the Bayes classifier in some asymptotic sense (as $n \rightarrow \infty$). 
A result of this type, namely that a family of classifiers approximates the Bayes classifier as $n \to \infty$, is known as an \emph{asymptotic consistency result}.

One of the key difficulties in briding the gap between the finite training sample and the unknown distribution $\bnu$ is balancing between \textit{overfitting} and \textit{underfitting}. When one constructs a very ``complex'' classifier so as to be faithful to the labels associated to the training data, it is said that the classifier overfits the data.  On the other hand, when one oversimplifies the classifier by sacrificing faithfulness to the observed data, it is said that the classifier underfits the data. The so called $1$-NN (one nearest neighbor) classifier is a typical example of a classifier that overfits: for a given $x\in D$ define the label of $x$ to be that of the point $\x_i$ closest to $x$. On the other hand, the classifier constructed by setting the label of every $x \in D$ to be the most common label among the training data, is the most extreme case of a classifier that underfits. Figure \ref{fig:over-under-example} shows examples of these situations. The natural question is thus: How does one construct an ``ideal'' classifier which neither overfits nor underfits a finite set of training data?
 \begin{figure}
     \centering
     \begin{subfigure}[b]{0.3\textwidth}
         \frame{\includegraphics[width=\textwidth]{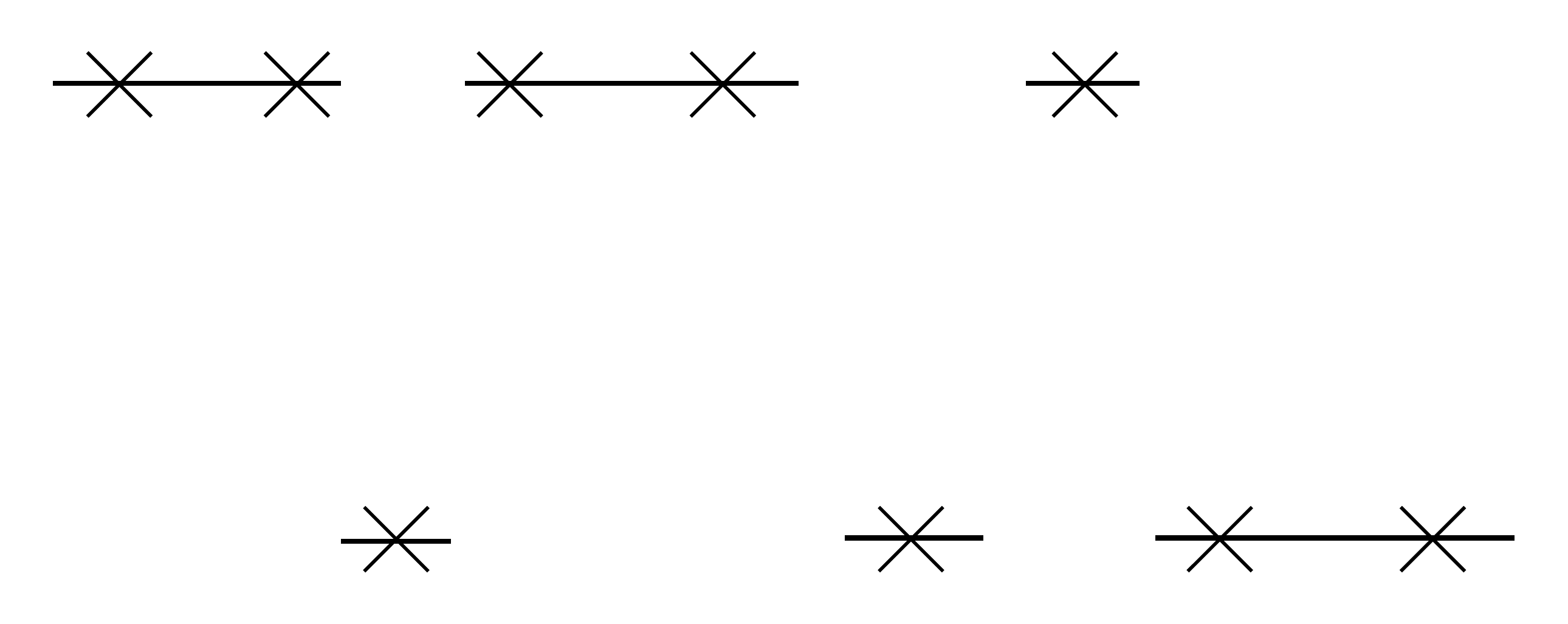}}
         %\caption{An overfitting classifier, namely $u_\1^n$.}
         \label{fig:overfit}
     \end{subfigure}
     ~ %add desired spacing between images, e. g. ~, \quad, \qquad, \hfill etc. 
     \begin{subfigure}[b]{0.3\textwidth}
         \frame{\includegraphics[width=\textwidth]{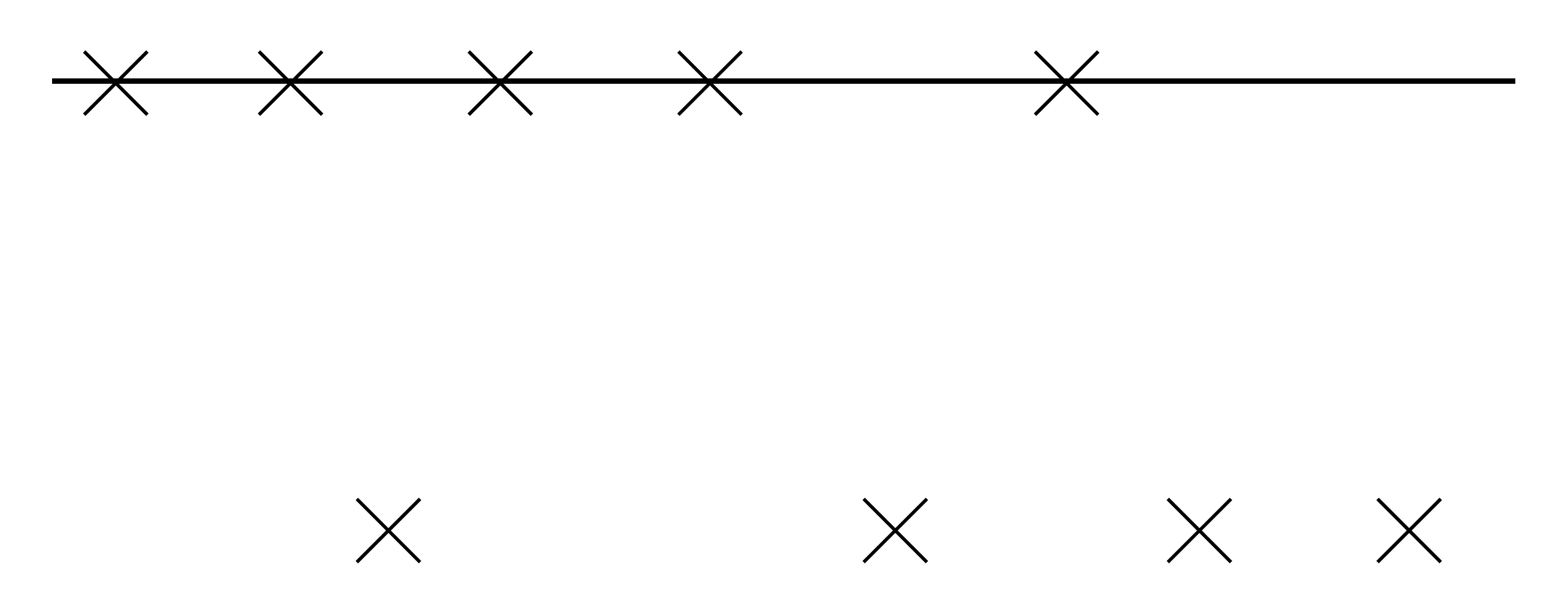}}
         %\caption{An underfitting classifier, namely picking the most common label.}
         \label{fig:underfit}
     \end{subfigure}
     ~ %add desired spacing between images, e. g. ~, \quad, \qquad, \hfill etc. 
     \begin{subfigure}[b]{0.3\textwidth}
         \frame{\includegraphics[width=\textwidth]{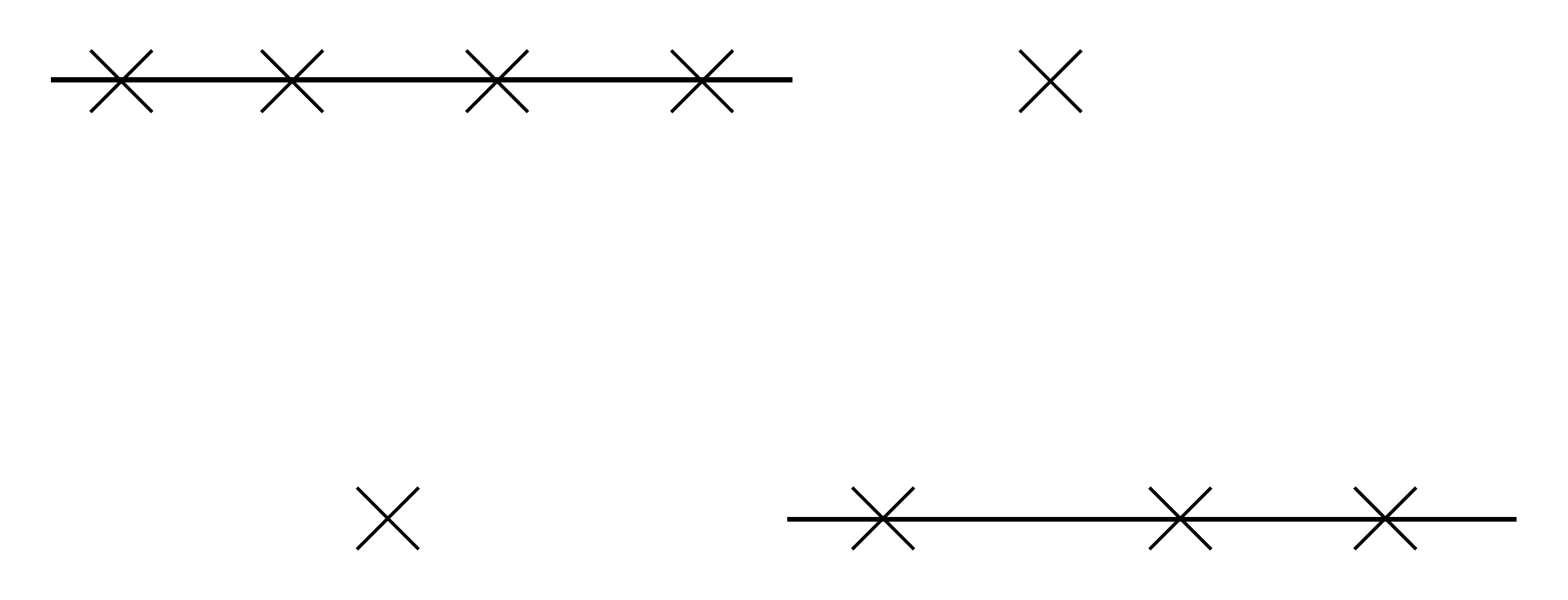}}
         %\caption{The Bayes classifier}
         \label{fig:bayes}
     \end{subfigure}
     \caption{Three different classifiers for a family of data points; the $x$-axis represents location and  the $y$-axis represents the labels $0$ or $1$. The first classifier, namely the nearest neighbor classifier $u_\1^n$, overfits the data. The second classifier picks the most common label, and underfits the data. The third classifier is the Bayes classifier.}\label{fig:over-under-example}
 \end{figure}

To answer the previous question one needs a clear mathematical notion of overfitting and underfitting. One central purpose of this paper is to give precise definitions for overfitting, underfitting, and consistency as asymptotic notions ($n \rightarrow \infty$) in a concrete analytical setting introduced in Subsection \ref{sec:set-up}.  

Before we describe our setting, it is helpful to consider the 1-NN classifier so as to get a better understanding of the problem of overfitting and the classical approaches to mitigating the same.  Let $l_n: \{ \x_1, \dots, \x_n\} \rightarrow \{ 0,1\}$ be the label function defined by $l_n(\x_i) = \y_i$.  The $1$-NN classifier, $u_{\1}^n$, is constructed by extending the function $l_n$, which is only defined on the point cloud, to the whole domain $D$ as described earlier. Since the labels $\y_i$ are random variables given $\x_i$, the function $l_n$ may take very different values at neighboring $\x_i$ and $\x_j$. The highly oscillatory nature of $l_n$ means that as $n \rightarrow \infty$ the function $l_n$ may not resemble any function $u$ defined on the whole domain $D$. The function $l_n$ will instead resemble a distribution, where at each point $x \in D$ one may have the value $1$ with certain probability and the value $0$ with certain probability.  In the language of modern analysis, we do not have compactness in the space of measurable functions, but instead in the space of Young measures. However, each classifier $u_{\1}^n$ is a function that when restricted to the training data coincides with the label function $l_n$. In particular, it minimizes the empirical risk, which for a function $u: D \rightarrow \R$ is defined as
\[  R_n(u) : = \frac{1}{n } \sum_{i=1}^{n} \lvert u(\x_i) - \y_i \rvert  = \frac{1}{n } \sum_{i=1}^{n} \lvert u(\x_i) - l_n(\x_i) \rvert. \]
Thus, if one seeks to construct a classifier via unconstrained empirical risk minimization then even basic properties, such as being a function, may be lost in the limit. This is partly due to the limitation that the functional $R_n$ is truly a functional defined for functions on the point cloud: $u_n : \left\{ \x_1, \dots, \x_n\right\} \rightarrow \R$.

Classically, the main approach for avoiding the problem of overfitting is to restrict, either explicitly or implicitly, the family of classifiers considered when trying to minimize the empirical risk $R_n$. After a family $\mathcal{F}$ of functions is specified, one must then prove asymptotic consistency, usually obtained by analyzing the variance and the bias associated to $\mathcal{F}$. One of the first main theoretical tools developed for the purpose of analyzing the variance is VC (Vapnik--Chervonenkis) theory. In VC theory, the shattering number $\mathcal{N}(\mathcal{F},n)$ of a family of functions $\mathcal{F}$ is defined by
\[
\mathcal{N}(\mathcal{F},n) := \max_{(x_i)_{i=1\dots n}} |\mathcal{F}_{(x_i)}|,
\]
where $\mathcal{F}_{(x_i)}$ is the restriction of the functions in $\mathcal{F}$ to the set $(x_i)$ and $|\mathcal{F}_{(x_i)}|$ is the number of distinguishable elements in $\mathcal{F}_{(x_i)}$. In essence, the shattering number gives one relevant measure of the capacity of the family of functions $\mathcal{F}$ to overfit a set of data points. One of the central results in VC theory is that if 
\[
\frac{\log_2 \mathcal{N}(\mathcal{F},n)}{n} \to 0
\]
then the empirical risk $R_n$ converges in probability uniformly (over $\mathcal{F}$) towards $R$. VC theory, and its many extensions, provide a powerful tool for proving asymptotic consistency. However, in many situations estimating the shattering number of a class of functions can be a challenging combinatorial problem.

As stated, the shattering number is defined in terms of some explicit family of classifiers $\mathcal{F}$. However, it is also possible to implicitly restrict the family of classifiers by minimizing a regularized empirical risk function of the form
\[ \min_{u: D \rightarrow \R}    R_n(u) + \lambda \Omega(u), \]
where $\Omega$ is some functional measuring the complexity of the classifier $u$. For example, $\Omega$ may be some integral of $\nabla u$, i.e. a TV or Sobolev norm. In this setting $\lambda$ is known as a regularization parameter, which specifies a tradeoff between fidelity ($R_n$) and smoothness $(\Omega)$. In this context VC theory can still be applied to the family of functions $\mathcal{F} = \{u : \Omega(u) <C\}$ if suitable combinatorial estimates are satisfied. A helpful overview of some of the classical techniques used to prove consistency is \cite{vLS2011}, and a standard reference addressing some of these topics is \cite{Vapnik1998}.

The classical theory outlined previously is based on classifiers that are \textit{extrinsic} to the data, in the sense that in both cases one considers a notion of complexity of families of functions defined on the \textit{whole underlying domain $D$}. This approach is very powerful in many settings, but can be difficult to apply in practice. The extrinsic approach may also be challenging when information about $D$ is limited and one is forced to work with families of functions defined on the whole ambient space $\R^d$ which may not be tailored to the geometry of $D$.  In this paper we take a different point of view and consider an \textit{intrinsic} approach, namely we first seek to construct a suitable function defined on the data cloud. In particular, we focus on a \textit{regularized empirical risk minimization} problem of the form 
\begin{equation} \label{eqn:intro1}  \min_{u_n}  R_n(u_n) + \lambda \Omega_n(u_n),\end{equation}
where $u_n$ is a function taking values on the point cloud $\{\x_1, \dots, \x_n\}$ and $\Omega_n$ is a regularizer constructed from the point cloud. This paper specifically addresses the asymptotic behavior of minimizers of the above regularized empirical risk minimization problem when $\Omega_n$ is the \textit{graph total variation} defined in \eqref{GTV} below. This functional depends on the construction of a proximity graph based on the point cloud and a parameter $\veps$ which specifies the connectivity of the graph. 

In establishing a consistency result, we need a suitable metric for comparing functions on a point cloud, namely minimizers of \eqref{eqn:intro1}, with functions defined on all of $D \subset \R^d$, namely $u_B$. In particular, we utilize the  $TL^1(D)$ metric space introduced in \cite{GarciaTrillos2015} (see \eqref{TL1def} below for its definition).  The $TL^1(D)$ metric space turns out to be very useful when stating our definitions of (asymptotic) overfitting, underfitting and consistency for different asymptotic regimes of $\lambda$. We show that if the regularizer is too weak ($\lambda$ small with respect to $\veps$), then the minimizers of the regularized empirical risk, despite forming a Cauchy sequence in $TL^1(D)$, do not converge to an element in the metric space $TL^1(D)$. In the completion of this metric space, the limit can be interpreted as a distribution, or Young measure, and not a function: this is an overfitting regime.
If the regularizer is too strong ($\lambda$ not decaying to zero), then the minimizers obtained are too regular and in the limit ($TL^1(D)$-limit) one recovers a regular function; when $\lambda \rightarrow \infty$ one recovers the most extreme case of underfitting. Finally, there is an `ideal' scaling regime where one recovers the Bayes classifier $u_B$ in the limit: this is an asymptotic  consistency result.

We also provide a simple means of constructing a classifier $u: \R^d \to \{0,1\}$ from the minimizer of the problem \eqref{eqn:intro1}. To this end, define the Voronoi extension (or $1$-NN extension) of a function $u_n: \left\{\x_1, \dots, \x_n \right\} \rightarrow \left\{ 0,1\right\}$  by
\begin{equation}
  u_n^V (x) = \sum_{i=1}^{n} u_n(\x_i) \textbf{1}_{V_i^n}(x); 
\label{VoronoiExtension}
\end{equation}
where $V_i^n$ is the set of points in $D$ whose closest point among $\left\{ \x_1, \dots, \x_n\right\}$ is $\x_i$; this set is called the Voronoi cell of the point $\x_i$. In simple words, the label assigned to a point $x\in D$ is the value of $u_n$ at its closest neighbor in the set $\left\{\x_1, \dots, \x_n \right\}$. The last theorem in this work proves that the Voronoi extensions of the minimizers of \eqref{eqn:intro1} indeed converge to the Bayes classifier when $\lambda$ scales appropriately.

In summary, we decompose the process of constructing a classifier into two steps. The first step involves solving a discrete, convex optimization problem, namely finding a minimizer of \eqref{eqn:intro1}.  The second step involves extending the minimizer via the Voronoi extension. This process is \textit{intrinsic} in the sense that it assumes no a priori information about the distribution, and uses only information derived from the point cloud. 
%This viewpoint is depicted in Figure \ref{fig:extrinsic-vs-intrinsic}. 

% 
% \begin{figure}
%     \centering
%     \begin{subfigure}[b]{0.6 \textwidth}
%     		\frame{\includegraphics[width=\textwidth]{./Figures/Extrinsic.pdf}}
%         %\caption{An overfitting classifier, namely $u_\1^n$.}
%         \label{fig:extrinsic}
%     \end{subfigure}
%     
%     ~ %add desired spacing between images, e. g. ~, \quad, \qquad, \hfill etc. 
%       %(or a blank line to force the subfigure onto a new line)
%     \begin{subfigure}[b]{0.9 \textwidth}
%     		\frame{\includegraphics[width=\textwidth]{./Figures/Intrinsic.pdf}}
%         %\caption{An underfitting classifier, namely picking the most common label.}
%         \label{fig:intrinsic}
%     \end{subfigure}
%     ~ %add desired spacing between images, e. g. ~, \quad, \qquad, \hfill etc. 
%     %(or a blank line to force the subfigure onto a new line)
%     \caption{A depiction of the standard \textit{extrinsic} approach to empirical risk minimization, and the \textit{intrinsic} approach proposed here.}\label{fig:extrinsic-vs-intrinsic}
% \end{figure}

There are several noteworthy features of this approach. First, the (limiting) family of classifiers attainable by this method is very broad, namely the family of $BV$ classifiers. In other words, the structural assumptions on the limit are quite weak, giving the method significant flexibility. Second, very little information is required about the initial distribution $\bnu$. In particular no information is needed about the support of $\bnu$, besides it being supported on some open, sufficiently regular set. The case in which $\bnu$ is supported on an embedded submanifold $\M \subseteq \R^d$ (with lower intrinsic dimension) can be addressed with similar techniques, but we will present the details elsewhere.

Our analytical framework differs from the classical learning theory approach 
in two main aspects. First,  \textit{regularity} of a minimizer of the functional \eqref{eqn:intro1} is enforced by the  $\Omega_n$ term and an appropriate choice of the parameter $\lambda_n$. In turn, this regularity guarantees the needed compactness in the appropriate metric space so as to guarantee the asymptotic consistency and avoid overfitting. 
Second, we directly compare minimizers of the empirical energies with minimizers of the analogous continuum (population level) energies, as opposed to studying only bounds on energy differences. Our point of view is amenable to analysis using transparent, modern tools from mathematics. These tools can be used both to prove important theoretical results, such as the consistency result of this paper, as well as to provide new insights into certain phenomena.  For example, the metric that we use in this paper provide clear means for defining asymptotic notions of over and underfitting. In particular, overfitting can be seen in terms of a loss of compactness, or convergence towards a non-trivial Young measure.

\subsection{Set-up}\label{sec:set-up}
To start developing the ideas presented in  the introduction, we first need to be more precise about the notions and assumptions we consider in this paper. 

Let $D \subseteq \R^d$ be a bounded, connected, open set with Lipschitz boundary. We measure the distance between two elements in $D$ with the Euclidean distance in $\R^d$.

We let $\nu$, the distribution of features, be given by $d \nu = \rho dx$, where $\rho: D \rightarrow \R$ is a continuous density function defined on $D$. We will assume that $\rho$ is bounded above and below by positive constants, that is, we assume that there are constants $0<m,M$ such that
\begin{equation}
m \leq \rho(x) \leq M , \quad \forall x \in D. 
\label{BoundsRho}
\end{equation}

We let $bnu$, the joint distribution of features and labels, be given by a Borel probability measure on $\R^d \times \R$ whose support is contained in $\overline{D} \times \left\{0,1 \right\} $ and whose first marginal is $\nu$. That is, 
for every Borel set $A \subseteq \R^d$, 
\[ \bnu(A \times \left\{0,1 \right\}) = \nu(A \cap D) = \int_{A \cap D} \rho(x) dx.   \]

For a random variable $(\x, \y)$ distributed according to $\bnu$, we let $\bnu_x$ be the conditional distribution of $\y$ given $\x=x$. That is, we use the disintegration theorem to write $\bnu$ as
\[  \bnu(A \times I) = \int_A\left(\int_{I} d \bnu_x(y)  \right) d \nu(x), \]
for all $A$ Borel subset of $D$ and for every interval $I \subseteq \R$. Expressed simply, $\bnu_x$ represents the distribution of labels of an object/individual with features $\x=x$. 

We let $\mu: D  \rightarrow \R$ be the conditional mean function, defined by
\begin{equation}
\mu(x):= \int_{\{0,1\}} y d\bnu_x(y) = \bnu_x( \left\{ 1 \right\} ) = \mathbb{P}(\y=1 | \x =x ).
\label{mu}
\end{equation}
The Bayes classifier $u_B: D \rightarrow \R$ is defined by
\begin{equation}
u_B(x):=
\begin{cases}
1 , \quad \text{if } \mu(x) \geq 1/2 \\
0, \quad \text{otherwise}. 
\end{cases}
\label{BayesClassifier} 
\end{equation}

It is straightforward to check that $u_B$ is a minimizer over $L^1(\nu)$ of the risk functional  
\begin{equation}
R(u) := \int_{D \times \R}|u(x) - y| d\bnu(x,y)  = \int_{D} \left( \int_{\R } \lvert  u(x) -  y \rvert     d \bnu_{x}(y) \right) d \nu(x),
\label{TrueRisk}
\end{equation}
where $L^1(\nu)$ is the space of real-valued functions integrable with respect to the measure $\nu$.

For ease of presentation, it will be desirable for $u_B$ to be the unique minimizer of $R$. To this end, observe that on the set $\{x \in D : \mu(x) = 1/2\}$, we may modify $u(x)$ to take any value in $[0,1]$ without increasing the value of $R$. Thus for $u_B$ to be unique, it is necessary to assume that
\begin{equation}\label{eqn:mu-neq-half}
 \nu \left( \left\{x \in D \: : \:  \mu(x)\not = 1/2 \right\} \right) =1.
 \end{equation}
In light of \eqref{BoundsRho}, this is equivalent to the statement $\mu \neq 1/2$ Lebesgue-a.e.

This condition is in fact sufficient for $u_B$ to be the unique minimizer of the risk functional $R$ over the class of $L^1(\nu)$-functions. Indeed, suppose that $u$ minimizes $R$. It is clear that if the set where $u$ takes values not in $[0,1]$ has non-zero measure, then $u$ can not be a minimizer of $R$; hence $u$ takes values in $[0,1]$ only. 
Now, given that $u$ takes values in $[0,1]$ only, we can write:
\begin{align}
\begin{split}
R(u)&= \int_{D} \left( \int_{\R} \lvert u(x) - y   \rvert d \bm{\nu}_x(y)   \right) d \nu(x)
\\& = \int_{D} \left( \lvert u(x) - 1 \rvert \mu(x) + \lvert  u(x) \rvert (1- \mu(x))      \right) d \nu(x) 
\\ & =  \int_{D} \left( ( 1-u(x))  \mu(x)+   u(x)  (1- \mu(x))      \right) d \nu(x)
\\ &=  \int_{D} \left( ( 1-u(x))  \mu(x)+   u(x)  (1- \mu(x))      \right) d \nu(x)
\\&=\int_{D}\mu(x) d \nu(x) + \int_{D}(1- 2 \mu(x))u(x) d \nu(x).
\end{split}
\end{align}
Now, by the definition of $u_B$, for any $u(x)$ only taking values in $[0,1]$ we have that $(1- 2 \mu(x)) u(x) \geq (1- 2 \mu(x)) u_B(x)$ for all $x \in D$. Under the assumption \eqref{eqn:mu-neq-half} this inequality can only be an equality at $\nu$ a.e. $x$ if $u = u_B$.
From this it follows that $R$ has a unique minimizer (the Bayes classifier) if and only if the set of $x$ with $\mu(x)=1/2$ has $\nu$-measure zero.

In addition to assumption \eqref{eqn:mu-neq-half}, which guarantees the uniqueness of minimizers for $R$, we also assume that $\nu( \left\{ x \in D \: : \: u_B(x) =1\right\})  \not = 1/2$, or in other words that the Bayes classifier has only one median.
We denote by $u^\infty$ the median of $u_B$, that is, 
\begin{equation}
u^\infty :=  \begin{cases} 1 & \text{ if }  \nu( \left\{ x \in D \: : \: u_B(x) =1\right\}) > 1/2  \\ 0 & \text{ otherwise. }
\end{cases}
\label{uinfty}
\end{equation}
It is then straightforward to check that $u^\infty$ is the unique minimizer of $\min_{ y \in \R} R(y)$.

We additionally make some weak regularity assumptions on the functions $\mu$ and $u_B$. We assume that the function $\mu$ is continuous at $\nu$-a.e. $x \in D$. 
In particular, $\mu$ is allowed to have discontinuities as long as the set at which $\mu$ is discontinuous is $\nu$-negligible. This assumption models the continuity of the law of $\y$ given that $\x=x$, as $x$ changes.
Also, we assume that $u_B$ is a function with finite total variation (we recall the definition of total variation in \eqref{TV}). We notice that the assumption on the regularity of the Bayes classifier, that is the regularity of the interface between the regions where $u_B=1$ and $u_B= 0$, is very mild. Specifically it only requires that the interface has finite perimeter; the notion of perimeter we use is that of Caccioppoli (see \cite{AFP}). 

Now let us consider $(\x_1, \y_1), \dots, (\x_n , \y_n)$ i.i.d. samples from $\bnu$. These are the training data representing $n$ objects/individuals with features $\x_i$ and corresponding labels $\y_i$. We denote by $\bnu_n$ the empirical measure
\[ \bnu_n:= \frac{1}{n}\sum_{i=1}^{n} \delta_{(\x_i,\y_i)}   \]  
and by $\nu_n$ the measure
\[ \nu_n := \frac{1}{n}\sum_{i=1}^{n}\delta_{\x_i}. \]
Observe that $\bnu_n $ is a measure on $D \times \R$ and $\nu_n$ a measure on $D$. 

The labels $\y_i$ define a \textit{label function} $l_n \in L^1(\nu_n)$, where 
$l_n :\left\{ \x_1 , \dots, \x_n \right\} \rightarrow \left\{0,1 \right\}$ and
\begin{equation}
 l_n(\x_i) := \y_i , \quad \forall i =1, \dots, n. 
 \label{labelfunc}
 \end{equation}
 In the above and in the remainder of the paper, $L^1(\nu_n)$ represents the space of integrable functions with respect to the measure $\nu_n$, i.e., real-valued functions whose domain is the set $\left\{\x_1, \dots, \x_n \right\}$.

Associated to the sample $(\x_1,\y_1), \dots, (\x_n, \y_n)$, we consider the \textit{empirical risk} functional $R_n : L^1(\nu_n) \rightarrow \R$ given by
\[  R_n( u_n) := \int_D \lvert u_n(x)- l_n(x) \rvert d \nu_n(x) = \frac{1}{n} \sum_{i=1}^{n} \lvert u_n(\x_i) - \y_i \rvert, \quad u_n \in L^1(\nu_n). \] 
We notice that the risk functional is intrinsic to the data, as it can be defined completely in terms of the values of $(\x_i, \y_i)$ for any arbitrary function $u_n \in L^1(\nu_n)$. We remark that if $u_n$ takes only values in $\left\{0,1 \right\}$, then $R_n(u_n)$ is simply the fraction of discrepancies between $u_n$ and the labels $\y_i$.  We also observe that using the empirical measure $\bnu_n$, the empirical risk functional $R_n$ may be written as
\[ R_n(u_n) = \int_{D \times \R} |u_n(x) - y| d\bnu_n(x,y)  .\]
When written in this form, we see that $R_n$ resembles the true risk \eqref{TrueRisk}. The main difference between $R_n$ and $R$ is that the argument of $R_n$ is a function $u_n \in L^1(\nu_n)$, whereas the argument of $R$ is a function $u \in L^1(\nu)$.

As we stated previously, the unique minimizer of the true risk functional \eqref{TrueRisk} is the Bayes classifier $u_B$ defined in \eqref{BayesClassifier}. On the other hand, it is evident that the function $l_n$ is the unique minimizer of the empirical risk $R_n$ among functions $u_n \in L^1(\nu_n)$. Despite the resemblance between $R_n$ and $R$, we can not expect to obtain $u_B$ as the limit of the functions $l_n$ in any reasonable topology.  As discussed in the introduction, this is due to the fact that the functions $l_n$ are ``highly oscillatory'' as $n \to \infty$, and hence can not converge to a function. To buffer the high oscillation of the functions $l_n$, while still being faithful to the labels $\y_i$,  one seeks to minimize a  risk functional with an extra ``regularizing'' term. To be more precise, we first consider a kernel $\eta: [0,\infty) \rightarrow [0, \infty)$ not identically equal to zero and satisfying the following assumptions:
\begin{itemize}
\addtolength{\itemsep}{2pt}
% \addtolength{\leftmargin}{60pt}
\item[\textbf{(K1)}] $\eta$ is non-increasing.
\item[\textbf{(K2)}] The integral $\int_{0}^{\infty} \eta(r) \, r^d dr $ is finite.
\end{itemize}
We note that the class of admissible kernels is  broad and includes
both Gaussian kernels and discontinuous kernels like one defined by
 $\eta$ of the form $\eta=1$ for $r \leq 1$ and $\eta=0$ for $r>1$.  
 The assumption (K2) is equivalent to imposing that the quantity 
 \begin{equation} \label{sigma_eta}
 \sigma_\eta:= \int_{\R^d} \eta(|h|)|h_1|dh,
\end{equation}
is finite, where $h_1$ is the first coordinate of the vector $h$. We refer to $\sigma_\eta$ as the \textit{surface tension} of the kernel $\eta$. Also, we will often use a slight abuse of notation and for a vector $h \in \R^d$ write $\eta(h)$ instead of $\eta(|h|)$.  

We make an additional assumption on $\eta$, namely,
\begin{equation}
\eta(r) \geq 1 , \quad \forall r \in [0,2].
\label{ExtraAssump}
\end{equation}
This assumption is mainly for convenience, since any kernel satisfying (K1) and (K2) can be rescaled to satisfy \eqref{ExtraAssump}.

Having chosen the kernel $\eta$, we choose $\veps>0$ and construct a weighted geometric graph with vertices $\{  \x_1, \dots, \x_n\}$; the parameter $\veps$ defines a length scale which determines the connectivity of the point cloud. The weights of this graph are given by
\[ W_{ij} := \eta_{\veps}(\x_i- \x_j), \]
where
\[\eta_{\veps}(z):=\frac{1}{\veps^d}\eta\left( \frac{z}{\veps} \right).\]
For a function $u_n \in L^1(\nu_n)$, namely a function whose domain is the vertices of the graph $(\{\x_n\},W)$, we define the \textit{graph total variation} by 
\begin{equation}
  GTV_{n , \veps}(u_n):= \frac{1}{ n^2 \veps^{d+1}}\sum_{i=1}^{n} \sum_{j=1}^{n}  \eta\left( \frac{\x_i - \x_j}{ \veps} \right)  \left| u_n(\x_i) - u_n(\x_j) \right|. 
\label{GTV}  
\end{equation}
The graph total variation was previously used in \cite{GarciaTrillos2015,TSvBLX_jmlr} in connection to approaches to clustering using balanced graph cuts.

In this work we will analyze the \textit{regularized empirical risk functional} given by
\begin{equation}
 R_{n, \lambda}(u_n) :=   \lambda GTV_{n , \veps} (u_n)    + R_n(u_n) , \quad  u_n \in L^1(\nu_n).  
\label{EmpRegRisk}
\end{equation}
Here $\lambda>0$ is a parameter whose role is to emphasize or deemphasize the effect of the regularizer $GTV_{n, \veps}$.
%{\red Having chosen $\veps>0$ ( i.e. having chosen the regularizer), we introduce a new parameter $\lambda>0$, whose role is to emphasize or deemphasize the effect of the regularizer $GTV_{n, \veps}$. More precisely, we consider the  \textit{regularized empirical risk functional} defined as
%\begin{equation}
% R_{n, \lambda}(u_n) :=   \lambda GTV_{n , \veps} (u_n)    + R_n(u_n) , \quad  u_n \in L^1(\nu_n).  
%\label{EmpRegRisk}
%\end{equation}
%}
We will generally assume that $\lambda$ and $\veps$ are allowed to vary as $n \to \infty$ (written $\lambda_n$ and $\veps_n$); this is natural in light of the results in \cite{GarciaTrillos2015}, which require specific decay rates on $\veps_n$.

The functional $R_{n , \lambda}$ is similar to the (ROF) model with $L^1$-fidelity term used in the context of image denoising (see \cite{Chambolle3,Nikolova}), but our setting and motivation is different 
from that in  \cite{Chambolle3,Nikolova}, as the functional $R_{n, \lambda}$ is constructed from a random sample $\{(\x_i,\y_i)\}_{i=1\dots n}$ of an unknown distribution $\bnu$. 
We remark that the $L^1$-fidelity term is well suited for the task of classification because it naturally generates functions valued in $\{0,1\}$, or, in other words, sparse functions. 
Numerical methods designed to find an approximate minimizer of \eqref{EmpRegRisk} can be found in \cite{Oman}; on the other hand an augmented Lagrangian approach to find the exact minimizer of \eqref{EmpRegRisk} can be found in \cite{Esser};
See also \cite{Chambolle3} and the references within.

The analogue of the functional $R_{n , \lambda}$ in the continuous setting is the functional
\begin{equation}
R_{\lambda}(u):= \lambda \sigma_\eta TV(u) + R(u) , \quad u \in L^1(\nu);
\label{RegRisk}
\end{equation}
where in the above, $TV$ denotes the (weighted by $\rho^2$) \textit{total variation} of the function $u \in L^1(\nu)$, which is defined by
\begin{equation}
 TV(u):= \sup \left\{ \int_{D} \divergence(\phi) u  dx \: : \:   \phi \in C^1_c(D : \R^d) , \text{ and } \lVert \phi(x) \rVert   \leq \rho^2(x) , \quad  \forall x \in D      \right\}. 
 \label{TV}
 \end{equation}
If the above quantity is finite, we say that $u \in L^1(\nu) $ is a function with bounded (weighted by $\rho^2$) variation. We have included the surface tension $\sigma_\eta$ in the definition of $R_\lambda$ in light of the results from \cite{GarciaTrillos2015} which state that $\sigma_\eta TV$ is the $\Gamma$-limit 
(we will make this precise in Theorem \ref{ContTV} below) of the functionals $GTV_{n, \veps}$, when $\veps$ scales with $n$ appropriately.

%The risk functional at this function is zero. From the machine learning point of view, an attempt to make a function coincide exactly with the given labels, 
%can be seen as the most extreme case of overfitting. From an analysis point of view, it is clear that the functions $l_n$ will not converge as $n \rightarrow \infty$ (in any reasonable way) 
%to the function $u_B$. We say in any reasonable way because both in a ''strong'' setting and in a weak ''setting'' the convergence will not occur. 
%The ''strong'' convergence is discarded due to the ''high oscillations'' associated to the functions $u_n$, which can be understood by observing that very similar features $\x_i, \x_j$, 
%may have completely different labels as the choice of labels are randomly chosen in an independent manner. The ''weak convergence'' on the other hand, which may be interpreted as convergence of averages,
%is also discarded, because if there was any reasonable limit of averages, that would be the mean function $\mu$ and not $u_B$. 

In order to state the main results of the paper, one needs a suitable metric for comparing functions in $L^1(\nu_n)$ with functions in $L^1(\nu)$. We consider the $TL^1$-metric space that was introduced in \cite{GarciaTrillos2015}.

We denote by $\mathcal{P}(D)$ the set of Borel probability measures on $D$. The set $TL^1(D)$ is defined as
\begin{equation}
 TL^1(D) := \{ (\theta, f) \; : \:  \theta \in \mathcal P(D), \, f \in L^1(D, \theta) \}. 
 \label{TL1def}
\end{equation}
That is, elements in $TL^1(D)$ are of the form $(\theta, f)$ , where $\theta$ is a probability measure on $D$ (in this paper we will take $\nu$ or $\nu_n)$, and $f \in L^1(\theta)$, that is $f$ is integrable with respect to $\theta$.
This space can be seen as a formal fiber bundle over $\mathcal P(D)$; the fibers are the different $L^1$-spaces corresponding to the different Borel probability measures over $D$. 

We endow $TL^1(D)$ with the metric
\begin{align} \label{tlpmetric}
\begin{split}
 d_{TL^1}((\theta_1,f_1), (\theta_2,f_2)) :=
 \inf_{\pi \in \Gamma(\theta_1, \theta_2)} \left( \iint_{D \times D} |x_1-x_2| + |f_1(x_1)-f_2(x_2)|  d\pi(x_1,x_2)  \right),
\end{split}
\end{align}
where $\Gamma(\theta_1, \theta_2)$ represents the set of couplings, or transportation plans between $\theta_1$ and $\theta_2$. That is, an element $\pi \in \Gamma(\theta_1, \theta_2)$ 
is a Borel probability measure on $D \times D$ whose marginal on the first variable is $\theta_1$ and whose marginal on the second variable is $\theta_2$. In \cite{GarciaTrillos2015} it is proved that $d_{TL^1}$ is indeed a metric.

Let us now discuss a characterization of $TL^1$-convergence of a sequence of functions $\left\{ u_n \right\}_{n \in \N}$  with $u_n \in L^1(\nu_n)$ towards a function $u \in L^1(\nu)$; we use this characterization in the remainder. 
We recall that a Borel map $T_n : D \rightarrow \left\{ \x_1, \dots, \x_n \right\}$ is said to be a transportation map between the measures $\nu$ and $\nu_n$, if for all $i$, $T_n^{-1}\left( \left\{ \x_i \right\}  \right)$ has $\nu$-measure equal to $1/n$.
The results from \cite{W8L8canadian}, imply that with very high probability, i.e. probability greater than $1- n^{- \beta}$ (for $\beta$ any number greater than one), there exists a transportation map $T_n$ between $\nu$ and $\nu_n$, such that
\begin{equation}
\lVert T_n - Id   \rVert_{L^\infty(\nu)} \leq \frac{C_\beta \log(n)^{p_d} }{n^{1/d}},
\label{NiceTranspMaps}
\end{equation}
where $p_d$ is a constant depending on dimension and is equal to $1/d$ for $d \geq 3$ and equal to $3/4$ when $d=2$; $C_\beta$ is a constant that depends on $\beta$, $D$ and the constants from \eqref{BoundsRho}. 
Notice that from Borell-Cantelli lemma and the fact that $\frac{1}{n^{\beta}}$ is summable, we can conclude that with probability one, we can find a sequence of transportation maps $\left\{ T_n \right\}_{n \in \N}$, such that for all large enough $n$, \eqref{NiceTranspMaps} holds. 
We refer the interested reader to \cite{W8L8canadian} for more background and references on the problem of finding transportation maps between some distribution and the empirical measure associated to samples drawn from it.

It is shown in \cite{GarciaTrillos2015} (see Proposition \ref{EquivalenceTLp} below) that $(\nu_n, u_n) \converges{TL^1} (\nu, u)$ if and only if $u_n\circ T_n \converges{L^1(\nu)} u$, where $T_n$ are the maps from \eqref{NiceTranspMaps} 
(which exist with probability one). We abuse notation a bit and simply say that $u_n \converges{TL^1} u$ in that case, understanding that $u_n \in L^1(\nu_n)$ and $u \in L^1(\nu)$.

\subsection{Main results}
\label{mainresults}

The first main result of this paper is related to the study of the limiting behavior of $u_n^*$ defined by:
\begin{equation}
  u_n^*:= \argmin_{u_n \in L^1(\nu_n)} R_{n, \lambda_n}(u_n),
  \label{ustar}
\end{equation}
under different asymptotic regimes for $ \left\{ \lambda_n \right\}_{n \in \N}$. %We remark (an it is straightforward to show it) that $u_n^*$ only takes values in $[0,1]$.

\begin{theorem}
Suppose that $(\x_1,\y_1), (\x_2,\y_2), \dots, (\x_n, \y_n), \dots $ are i.i.d. random variables distributed according to $\bnu$. Consider a sequence $\left\{ \veps_n \right\}_{n \in \N}$ satisfying

\begin{equation}
 \frac{ (\log(n))^{p_d}}{n^{1/d}}  \ll \veps_n \ll 1, 
 \label{vepsn}
 \end{equation}

where $p_d= 1/d$ when $d \geq 3$ and $p_2 = 3/4$.  Additionally, let $\left\{ \lambda_n \right\}_{n\in \N}$ be a sequence of positive real numbers. 

\begin{enumerate}
\item If $\lambda_n \ll \veps_n $ as $n \rightarrow \infty$ then, with probability one, $u_n^*= l_n$ for $n$ sufficiently large and $u_n^*$ does not converge in the $TL^1$-sense towards any function $u \in L^1(\nu)$. In addition, 
\[ \lim_{n \rightarrow \infty} R_{n}(u_n^*) =0.\]

\item If $\veps_n \ll \lambda_n \ll 1 $ as $n \rightarrow \infty$ then, with probability one, $u_n^*$ converges in the $TL^1$-sense towards the Bayes classifier $u_B$. In addition,
\[ \lim_{n \rightarrow \infty} R_{n}(u_n^*) = R(u_B).\]

\item If $ \lambda_n \rightarrow \lambda \in (0,\infty) $ as $n \rightarrow, \infty$ then, with probability one, every subsequence of $\left\{ u_n^* \right\}_{n \in \N}$ 
has a further subsequence that converges to a minimizer of $R_\lambda$ defined in \eqref{RegRisk}. In addition,
\[ \lim_{n \rightarrow \infty} R_{n , \lambda_n}(u_n^*) = \min_{u \in L^1(\nu)} R_{\lambda}(u). \]

\item If $\lambda_n \rightarrow \infty$ as $n \rightarrow \infty$ then, with probability one, $u_n^*$ converges in the $TL^1$-sense towards the constant function $u^\infty$ defined in \eqref{uinfty}. In addition,
\[ \lim_{n \rightarrow \infty} R_{n}(u_n^*) = \min_{y \in \R}R(y).      \]

\end{enumerate}  
\label{mainTheorem}
\end{theorem}

\begin{remark}
The conclusion of the theorem continues to hold even if the sequence $\left\{u_n^* \right\}_{n \in \N}$ is only assumed to be a sequence of almost minimizers of the energies $R_{n,\lambda_n}$.  That is, we only have to assume that
\[  \lim_{n \rightarrow \infty} \left( R_{n, \lambda_n}(u_n^*) -  \min_{u_n \in L^1(\nu_n)}R_{n, \lambda_n}(u_n) \right) =0   \]
for the conclusions of the theorem to be true.
\end{remark}

\begin{remark}
The assumption \eqref{vepsn} provides a natural setting under which the geometric graph is sufficiently well-connected. This was studied in detail in \cite{GarciaTrillos2015}.
\end{remark}

Theorem \ref{mainTheorem} provides a clear characterization of the asymptotic behavior of $u_n^*$ depending on the scaling of the parameter $\lambda_n$.

In the regime $ \veps_n \ll \lambda_n \ll 1 $, we obtain the Bayes classifier as the limit of the functions $u_n^*$ in the $TL^1$-sense. Here we find the balance between enough regularization (so that the limit of $u_n^*$ is a function)
and enough fidelity (so that the limit of $u_n^*$ is not just any function, but the Bayes classifier). We illustrate this regime in
Figure \ref{FigureConsistency}. In that example we have chosen  $D$ to be the unit square $(0,1)^2$ and the measure $\nu$ was chosen to be the uniform distribution on $D$. The function $\mu$ determining the conditional distribution of $\y$ given $\x=x$ was chosen to take two values $0.45$ and $0.55$; in the upper left corner and lower right corner $\mu=0.55$ whereas in the upper right corner and lower left corner $\mu=0.45$. A number of samples from the resulting distribution $\bm{\nu}$ are shown in Figure \ref{fig:Consistency1}. The function $u_n^*$ was constructed using the algorithm proposed in \cite{Esser}; in Figure \ref{fig:Consistency2} we present an appropriate level set of the function $u_n^*$.
 
 \begin{figure}
     \centering
     \begin{subfigure}[b]{0.5\textwidth}
         \includegraphics[width=\textwidth, trim =1cm 5cm 1cm 4cm]{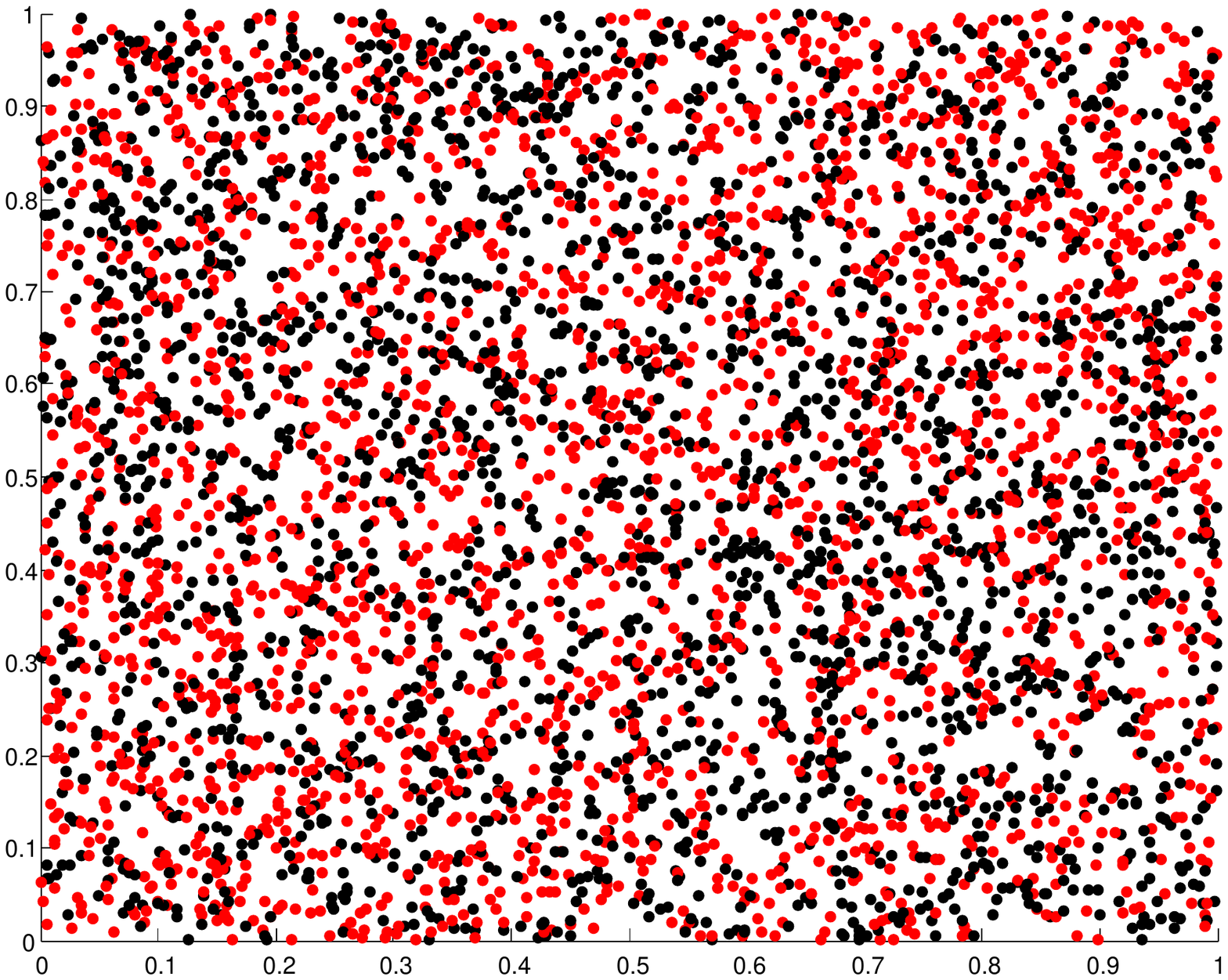}
         \caption{$n=10000$ random samples from $\bnu$.}
         \label{fig:Consistency1}
     \end{subfigure}
     ~
     \begin{subfigure}[b]{0.5\textwidth}
         \includegraphics[width=\textwidth, trim =1cm 5cm 1cm 4cm ]{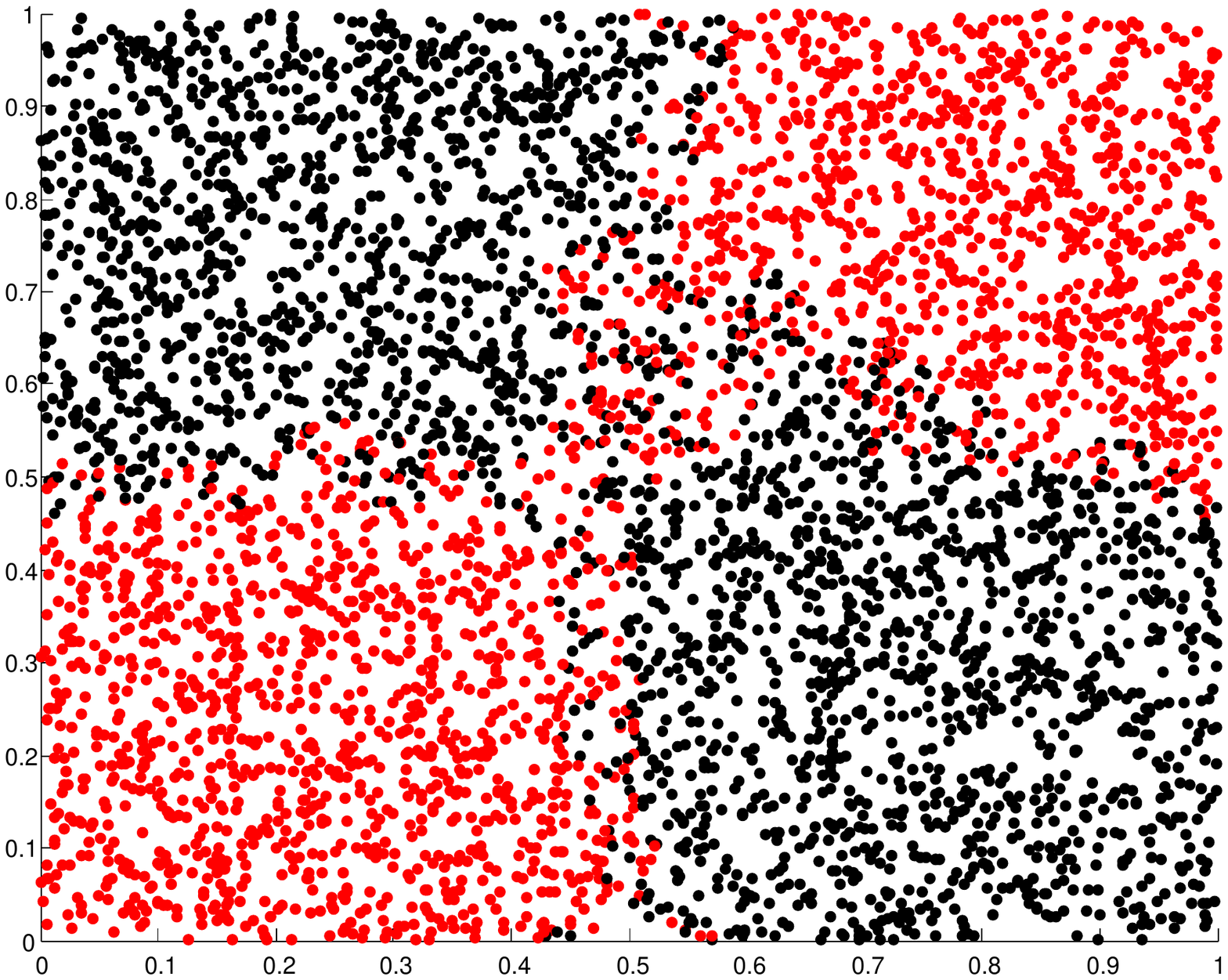}
         \caption{$u_{n}^*$ using $\veps= n^{-1/3}$ and $\lambda= n^{-1/4}$.}
         \label{fig:Consistency2}
     \end{subfigure}
     \caption{Example of consistency regime.}\label{FigureConsistency}
 \end{figure}

In the regime $\lambda_ n \ll \veps_n$, which we will call the \textit{overfitting regime}, the sequence of functions $u_n^*$ minimizing $R_{n, \lambda_n}$ does not converge to $u_B$ in the $TL^1$ sense, and in fact it does not converge to any function $u \in L^1(\nu)$. Instead, $u_n^*$, or in other words $l_n$, converges towards $\bnu$ in the completion of the $TL^1(D)$ space; see Subsection \ref{TLp} for a discussion regarding the completion of $TL^1(D)$. It is important to highlight that the limit of $u_n^*$ is not a function, but a measure (a Young measure more precisely). This type of limit is a consequence of using a regularizer term in the functional $R_{n, \lambda_n}$ that is not strong enough to control the oscillations of the label function $l_n$. In light of this, one could intuitively define overfitting as an asymptotic tendency towards Young measures.

When $\lambda_n \rightarrow \infty$, the functions $u_n^*$ approach  the constant function $u^\infty$ (the median of the Bayes classifier). We may view this regime as an \textit{underfitting regime}: the limit of the functions $u_n^*$ is a very regular function (a constant function) that is as faithful to the labels as possible given the strong regularity constraint. 

Finally, the regime $ \lambda_n \rightarrow \lambda \in(0,\infty)$, interpolates between the regime in which we recover $u_B$ and the regime in which we recover $u^\infty$. Indeed, in this regime we recover (up to subsequence) a function $u_\lambda$ minimizing the regularized risk functional $R_\lambda$ defined in \eqref{RegRisk}. For small values of $\lambda$, $u_\lambda$ should resemble the Bayes classifier, whereas for $\lambda$ large $u_\lambda$ should resemble $u^\infty$. 
This may be viewed as a weak underfitting regime,  which in the limit recovers a regularized version of the Bayes classifier.

Theorem \ref{mainTheorem} provides a type of consistency result for regularized empirical risk minimization as the sample size $n$ goes to infinity. Moreover, this consistency result gives a means of characterizing the statistical notions of overfitting and underfitting through modern analytical notions (such as loss of compactness and Young measures). In this particular case it is also possible to quantify precisely the notions of underfitting/overfitting by means of the asymptotic behavior of the sequence $\left\{ \lambda_n \right\}_{n \in \N}$.

However, at this stage, we have not truly addresed the classification problem. We have only given a means of constructing a suitable function $u_n^*$ defined on the geometric graph $(\{\x_n\}, W)$. Thus, the natural question at this stage is how to construct a ``good'' classifier using $u_n^*$.

Given the definition of $TL^1$ convergence, we know that there exists a family of transportation maps $T_n$ so that $u_n^* \circ T_n \to u_B$ in $L^1(\nu)$. However, without explicit knowledge of $D$ and $\nu$ it is not possible to construct the transport maps $T_n$. Thus we see that while the $TL^1$ space and the transportation maps $T_n$ are useful for the asymptotic analysis of the regularized empirical risk minimization problem, they \textit{do not} immediately build a bridge between such minimization problem and the problem of classification.

Fortunately,  it is possible to construct a good classifier from $u_n^*$  by simply considering  its Voronoi extension. We will show that these extensions converge under slightly less general assumptions than those from Theorem \ref{mainTheorem} towards the Bayes classifier. This is the content of our last main result.

\begin{theorem}
Suppose that $(\x_1,\y_1), (\x_2,\y_2), \dots, (\x_n, \y_n), \dots $ are i.i.d. random variables distributed according to $\bnu$. Consider a sequence $\left\{ \veps_n \right\}_{n \in \N}$ satisfying
\[  \frac{ (\log(n))^{p_d}}{n^{1/d}}  \ll \veps_n \ll 1, \] 
where $p_d= 1/d$ when $d \geq 3$ and $p_2 = 3/4$.  Additionally, let $\left\{ \lambda_n \right\}_{n \in \N}$ be a sequence of positive real numbers satisfying,
\[ (\log(n))^{d\cdot p_d}\veps_n \ll \lambda_n \ll  1. \] 

Then, with probability one,

\[ u_n^{*V} \converges{L^1(\nu)} u_B , \quad \text{ as } n \rightarrow \infty,\]

where $u_n^*$ is a minimizer of $R_{n,\lambda_n}$ and $u_n^{*V}$ is the Voronoi extension (as defined in \eqref{VoronoiExtension}) of $u_n^*$.\label{mainTheorem2}
\end{theorem}

The bottom line is that, for $\left\{\lambda_n \right\}_{n \in \N}$ chosen appropriately, it is possible to construct an ``intrinsic'' classifier which converges towards the Bayes classifier $u_B$. This is constructed by first finding $u_n^*$ using convex optimization, and then by extending using the Voronoi partition.

\begin{remark}
In general, it is unknown whether convergence in $TL^1$ is equivalent to convergence of Voronoi extensions. The work here (e.g. the proof of Theorem \ref{mainTheorem2}) suggests that this is at least plausible under certain regularity conditions.
In any case, we do not seek to address the question of the convergence of the Voronoi extensions of $u_n^*$ without the hypotheses in Theorem \ref{mainTheorem2}.
\end{remark}

\subsection{Discussion and future work}

Our work establishes the consistency of the empirical risk minimization problem \eqref{eqn:intro1} by showing that with the right choice of scaling for  $\lambda_n$, the minimizer $u_n^*$ converges towards the Bayes 
classifier in the $TL^1$-sense. Although the function $u_n^*$ is only defined on the cloud $\{ \x_1, \dots, \x_n \}$, one may extend the function $u_n^*$ in a simple way to the whole ambient space 
so as to obtain a classifier that in the limit converges towards the desired Bayes classifier. 
We remark that we do not use the notion of VC dimension explicitly in our analysis given that we do not consider classes of functions defined on the ambient space 
as feasible elements in the empirical risk minimization problem. Instead, we work directly with the graph and its natural space of functions; in our analysis we exploit the level of regularity of 
minimizers of $R_{n,\lambda_n}$ (enforced by the graph total variation) and we use the $TL^1$ distance to compare the solutions of the discrete problem with the Bayes classifier. 

We suspect a close connection between \textit{regularity} of a solution of a discrete problem like the one considered in this paper and the VC dimension of a certain implicit family of functions. A natural setting in which to investigate notions of regularity (along with their connection to VC theory) would be in the linear setting in which one attempts to minimize an energy of the form
\[   E_{n, \lambda_n}(u_n):= \frac{\lambda_n}{ n^2 \veps^{d+2}}\sum_{i=1}^{n} \sum_{j=1}^{n}  \eta\left( \frac{\x_i - \x_j}{ \veps} \right)  ( u_n(\x_i) - u_n(\x_j) )^2 + \frac{1}{n } \sum_{i=1}^{n} ( u_n(\x_i) - \y_i ) ^2, \quad u_n \in L^2(\nu_n), \]
with the goal of approximating the Bayes regressor $u(x):= \mathbb{E}(\y | \x=x)$, where the variable $\y$ follows a law of the form
\[ \y \sim \Prob(\y \in dy | \x=x ).\]
The minimizer of the energy $E_{n, \lambda_n}$ can be found by solving a linear system of equations involving the graph Laplacian associated to the graph $(\{ \x_i \} , W)$, which can be interpreted as an elliptic PDE on the graph. Appropriate analogs of techniques from elliptic theory, such as Schauder estimates and convex analysis, might then be powerful tools for analysis. We anticipate that these tools will permit a finer analysis of the problem, including detailed estimates on rates of convergence. The development of these tools, as well as their application, is the subject of current investigation.

Finally, we notice that the setting that we have considered in this paper is that in which the support of the measure $\nu$ is an open domain $D \subseteq \R^d$. It is natural to consider the case in which the support of $\nu$ is actually a sub-manifold $\M$ embedded in $\R^d$. We believe that the consistency results presented in this paper can be extended to the sub-manifold setting in a relative straightforward way. In the interest of clarity we defer the details to a later work. In the linear problem described above, we anticipate that the desired rates of convergence will depend only on geometric quantities of $\M$ and not on the ambient space $\R^d$.

\subsection{Outline}
The rest of the paper is organized as follows. In Section \ref{Prelim} we present preliminary results that we use in the remainder of the paper. 
Specifically, in Subsection \ref{TLp} we present some relevant properties of the $TL^1$ space and its completion; in Subsection \ref{AuxResul} we present the main results from \cite{GarciaTrillos2015} together with some other auxiliary results that we use in the remainder of the paper.  
In Section \ref{SecMainTheorem} we prove Theorem \ref{mainTheorem}; we do this in three steps: in Subsection \ref{OverRegime} we consider the overfitting regime; in Subsection \ref{UnderRegime} we consider the underfitting regime and finally in Subsection \ref{GoodRegime} we consider the intermediate regime where one obtains convergence towards the Bayes classifier. 
Finally, in Section \ref{SecMainTheorem2} we establish Theorem \ref{mainTheorem2}. 
\section{Preliminaries}
\label{Prelim}

\subsection{The metric space $TL^1$}
\label{TLp}

This section states some important properties of the $TL^1$ space.

To begin, we demonstrate that $(TL^1(D), d_{TL^1})$ is a metric space. This is accomplished by identifying the set $TL^1(D)$  with a subset of a space of probability measures over $D \times \R$  and by identifying the metric $d_{TL^1}$ with the \textit{earth mover's distance} over such space of measures. 

In order to develop this idea, denote by $\mathcal{P}_1(\overline{D} \times \R)$ the set of Borel probability measures whose support is contained in $\overline{D} \times \R$ and that have finite first moments, that is $\bm{\theta} \in \mathcal{P}(\overline{D} \times \R)$ belongs to $ \mathcal{P}_1( \overline{D} \times \R)$ if 
\[  \int_{D \times \R} (|x| + |y|) d \bm{\theta}(x, y) < \infty.   \]
The \textit{earth mover's distance} between two elements $\bm{\theta}_1, \bm{\theta}_2 \in \mathcal{P}_1(\overline{D} \times \R)$ is defined by:
\[ d_{1}(\bm{\theta}_1, \bm{\theta}_2) :=   \inf_{\bm{\pi} \in \Gamma(\bm{\theta}_1, \bm{\theta}_2)} \iint_{(D\times\R) \times (D \times \R)} (|x_1-x_2| + |y_1-y_2| ) d\bm{\pi}(x_1,y_1,x_2,y_2).\]  

Now, given a measure $\theta \in \mathcal{P}(D)$ and a Borel map $\mathbb{T}: D \rightarrow D \times \R$, define the \textit{push forward} of $\theta$ by $\mathbb{T}$ as the measure 
$\mathbb{T}_{\sharp} \theta $ in $\mathcal{P}(D \times\R)$ defined by
\[ \mathbb{T}_{\sharp} \theta ( A \times I) =  \theta \left( \mathbb{T}^{-1}( A \times I ) \right), \quad \forall A \subset D \text{ Borel }, \quad \forall I \subset \R \text{ Borel }.    \]

With the previous definitions in hand, we may now identify elements in $TL^1(D)$ with probability measures in $\mathcal{P}_1(\overline{D}\times\R)$ using the map
\begin{equation}
(\theta, f) \in TL^1 \longmapsto (Id \times f)_\sharp \theta \in \mathcal{P}_1(\overline{D}\times\R), 
\label{embedding}
\end{equation}
where $Id \times f$ is the map $x \in D \mapsto (x,f(x)) \in D \times \R $. In other words, $(\theta, f)$ is identified with a measure supported on the graph of the function $f$. 
Notice that indeed $(Id \times f)_\sharp \theta$ has first integrable moments, due to the boundedness of the set $D$ and the fact that $f \in L^1(D, \theta)$. Furthermore, $d_{TL^1}( (\theta_1, f_1), (\theta_2,f_2)   ) = d_1( (Id \times f_1)_\sharp \theta_1, d_{1}((Id \times f)_\sharp \theta) )$ 
for any two elements $ (\theta_1, f_1), (\theta_2,f_2) \in TL^1(D)$ (see \cite{GarciaTrillos2015}). That is, the map \eqref{embedding} is an isometric embedding of $TL^1(D)$ into $\mathcal{P}_1(\overline{D} \times \R)$.  

A simple example suffices to demonstrate that $(TL^1(D), d_{TL^1})$ is not a complete metric space.
\begin{example}
Let $D = (0,1)$, $\theta$ be the Lebesgue measure and $f_{n+1} := \sign \sin(2^n \pi x)$ for $x \in (0,1)$. By constructing transport maps that swap neighboring regions valued at $\pm 1$, it can be shown that $d_{TL^1}((\theta,f_n), (\theta,f_{n+1})) \leq 1/2^n$. This implies that the sequence $\left\{ (\theta, f_n) \right\}_{n \in \N}$ is a Cauchy sequence in $(TL^1(D), d_{TL^1})$. 
However, if this was a convergent sequence it would have to converge to an element of the form $(\theta, f)$ (see Proposition \ref{EquivalenceTLp} below), but then, by Remark \ref{remarkTLpandLp}, it would be true that $f_n \converges{L^1(\theta)} f$.  This is impossible because $\left\{f_n \right\}_{n \in \N}$ is not a convergent sequence in $L^1(D,\theta)$.
\end{example}

The previous example illustrates the idea that highly oscillating functions (in this case the functions $f_n$) do not converge to any element of $TL^1(D)$. On the other hand, since $\{ (\theta, f_n)\}$ was a Cauchy sequence, it will converge in the completion of $TL^1(D)$. In fact, we can actually interpret the limit as a \textit{Young measure} or \textit{parametrized measure} (see \cite{EvansWeak,CalcVarLeoni,Pedregal2}). Young measures are a type of generalized function, which associate each point $x \in D$ with a probability measure $\eta_x$ over $\R$. In the example presented above, the Young measure obtained in the limit is $\eta_x = 1/2 \delta_{-1} + 1/2 \delta_1$. Young measures can naturally be associated with elements of $\mathcal{P}_1(\overline{D} \times \R)$. 
We claim that the space $(\mathcal{P}_1(\overline{D} \times \R), d_{1})$ is the completion of $TL^1(D)$. 
To see this, first note that $TL^1(D)$ can be embedded isometrically into $\mathcal{P}_1(\overline{D} \times \R)$. Second, note that $(\mathcal{P}_1(\overline{D} \times \R), d_{1})$ is a complete metric space (see \cite{ambrosio2008gradient}). 
Finally, it is shown in \cite{GarciaTrillos2015} that $TL^1(D)$ is dense in $\mathcal{P}_1( \overline{D} \times \R)$. From the previous facts the claim follows.

\nc

%We also remark that one could also change the set and consider a metric where the powers of the terms in \eqref{tlpmetric} would be different ($p$ and $q$, instead of $p$ and $p$ and the natural name for the space in this case would be $TL^{p,q}$).

% \nc
% 
% 
% \begin{remark}
% If one restricts the attention to measures $\mu, \nu \in \mathcal{P}(D)$ which are absolutely continuous with respect to the Lebesgue measure then 
% \[ \inf_{T \::\: T_\sharp \mu = \nu} \left(\int_D |x - T(x)|^p + |f(x) - g(T(x))|^p d \mu(x) \right)^{\frac{1}{p}} \]
%  majorizes $d_{TL^p}((\mu,f), (\nu,g)) $ and furthermore provides a metric (on the subset of $TL^p$)
%  which gives the same topology as $d_{TL^p}$.  The fact that these topologies are the same follows from 
% Proposition \ref{EquivalenceTLp}.
% \end{remark}
% 

After discussing the $TL^1$-space and its completion, we state a useful characterization of $TL^1$-convergence. From this characterization, we see{\blue , } in particular, that the $TL^1$ 
convergence extends simultaneously the notion of (strong) convergence in $L^1$, and the notion of weak convergence (in fact, convergence in the earth mover's distance sense) of probability measures in $\mathcal{P}(D)$.

Let us first recall that given two measures $\theta_1 , \theta_2 \in \mathcal{P}(D)$, a Borel map $T: D \rightarrow D $ is a \textit{transportation map} between $\theta_1$ and $\theta_2$, if $\theta_2 = T _{\sharp } \theta_1$, where 
$T _{\sharp } \theta_1$ is the push forward of the measure $\theta_1$ by $T$. That is, $T$ is a transportation map between $\theta_1$ and $\theta_2$ if
\[  \theta_2(A) = \theta_1(T^{-1}(A)), \quad \text{ for all Borel } A \subseteq D.\]
A useful property of transportation maps is the change of variables formula:
\begin{equation}
\int_{D} f(T(x)) d \theta_1(x) = \int_{D}f(z)d \theta_2(z).    
\label{chvar}
\end{equation}
which holds for every Borel function $f: D \rightarrow \R$. This formula follows directly from the definition of transportation maps and an approximation procedure using simple functions.  

The following characterization can be found in \cite{GarciaTrillos2015}.

\begin{proposition}[Characterization of $TL^1$-convergence]
Let $(\theta,f) \in TL^1(D)$ and let
 $\left\{ \left(\theta_n , f_n \right) \right\}_{n \in \N}$ be a sequence in $TL^1(D)$. The following statements are equivalent:
\begin{enumerate}[(i)]
\item  $ \left(\theta_n , f_n \right)  \converges{TL^1} (\theta, f)$ as $n \rightarrow \infty$.
\item $\theta_n \overset{w}{\longrightarrow} \theta$ (to be read $\theta_n$ converges weakly towards $\theta$) and for every sequence of transportation plans $\left\{\pi_n \right\}_{n \in \N}$
(with $\pi_n \in \Gamma(\theta, \theta_n)$) satisfying
\begin{equation}
 \lim_{n \rightarrow \infty}\int \lvert x- y \rvert d \pi_n(x,y)=0 
\label{Stagplans}
\end{equation}
we have:
\begin{equation}
\iint_{D \times D} \left| f(x) - f_n(y) \right| d\pi_n(x,y) \rightarrow 0, \:  as \: n \rightarrow \infty.
\label{convergentTLp}
\end{equation} 
\item $\theta_n \overset{w}{\longrightarrow} \theta$ and there exists a sequence of transportation plans $\left\{\pi_n \right\}_{n \in \N}$ (with $\pi_n \in \Gamma(\theta, \theta_n)$) satisfying \eqref{Stagplans} for which \eqref{convergentTLp} holds.
\end{enumerate}
Moreover, if the measure $\theta$ is absolutely continuous with respect to the Lebesgue measure, the following are equivalent to the previous statements:

\begin{enumerate}[(i)]
  \setcounter{enumi}{3}
%\addtocounter{broj1}{3}
\item $\theta_n \overset{w}{\longrightarrow} \theta$ and for every sequence of transportation maps $\left\{ T_n \right\}_{n \in \N}$  (with ${T_n}_\sharp \theta = \theta_n$) satisfying
\begin{equation}
 \lim_{n \rightarrow \infty}\int \lvert T_n(x) - x  \rvert d \theta(x)=0 
\label{Stagmaps}
\end{equation}
we have
\begin{equation}
\int_{D} \left| f(x) - f_n\left(T_n(x)\right) \right| d\theta(x) \rightarrow 0, \:  as \: n \rightarrow \infty.
\label{convergentTLpMap}
\end{equation}
\item  $\theta_n \overset{w}{\longrightarrow} \theta$ and there exists a sequence of transportation maps $\left\{ T_n \right\}_{n \in \N}$ (with ${T_n}_\sharp \theta = \theta_n$) satisfying \eqref{Stagmaps} for which \eqref{convergentTLpMap} holds.
\end{enumerate}

\label{EquivalenceTLp}
\end{proposition}

The previous result allows us to abuse notation and talk about convergence of functions in $TL^1$ without having to specify the measures they are associated to. More precisely, suppose that the sequence
$\left\{ \theta_n  \right\}_{n \in \N} $ in $\mathcal{P}(D)$ converges weakly to $\theta \in \mathcal{P}(D)$. We say that the sequence $\left\{ u_n\right\}_{n \in \N}$ (with $u_n \in L^1(\theta_n)$) converges in the $TL^1$ sense to $u \in L^1(\theta)$, if $\left\{  \left(\theta_n , u_n  \right) \right\}_{n \in \N}$ converges to $(\theta, u)$ in the $TL^1$ metric space.
In this case we write $u_n \overset{{TL^1}}{\longrightarrow} u$ as $n \rightarrow \infty$. Also, we say that the sequence $\left\{ u_n\right\}_{n \in \N}$ (with $u_n \in L^1(\theta_n)$) is relatively compact in $TL^1$ if the sequence $\left\{   \left(\theta_n , u_n  \right)\right\}_{n \in \N}$ is relatively compact in $TL^1$.
In the remainder of the paper, we use the previous proposition and observation as follows: we let $\theta_n = \nu_n$ (the empirical measure associated to the samples from the measure $\nu$) and let $\theta= \nu$; 
we know that with probability one $\nu_n \converges{w} \nu$. We also know that with probability one, the maps from \eqref{NiceTranspMaps} exist and so for a sequence of functions $ \left\{ u_n \right\}_{n \in \N}$ with $u_n \in L^1(\nu_n)$,
we can say $u_n \converges{TL^1} u$ for $u \in L^1(\nu)$ if and only if $u_n\circ T_n \converges{L^1(\nu)} u$. Notice that this was the characterization used right before stating Theorem \ref{mainTheorem}.

\begin{remark}
\label{remarkTLpandLp}
We finish this section by noticing that from Proposition \ref{EquivalenceTLp}, we can think of the convergence in $TL^1$ as a generalization of weak convergence of measures and of $L^1$ convergence  of functions. 
That is $\left\{ \theta_n \right\}_{n \in \N}$ in $\mathcal{P}(D)$ converges weakly to $\theta\in \mathcal{P}(D)$ if and only if $ \left(\theta_n , 1 \right) \overset{{TL^1}}{\longrightarrow} (\theta, 1)$  as $n \rightarrow \infty$;
 and that for fixed $\theta \in \mathcal{P}(D)$ a sequence $\left\{ f_n \right\}_{n\in \N}$ in $L^1(\theta)$ converges in $L^1(\theta)$ to $f$ if and only if $(\theta, f_n) \overset{{TL^1}}{\longrightarrow} (\theta,f)$ as $n \rightarrow \infty$.
\end{remark}

\subsection{Auxiliary properties and results}
\label{AuxResul}

We now present the following additional properties that, as we will see, prove to be useful when establishing the main results of the paper. 

Given a sequence $\left\{ u_n \right\}_{n \in \N}$ with $u_n \in L^1(\nu_n)$, we say that $\left\{ u_n \right\}_{n \in \N}$ \textit{converges weakly} to $u \in L^1(\nu)$ (and denote this convergence by $u_n \rightharpoonup u$) if the sequence of functions $\left\{ u_n \circ T_n \right\}_{n \in \N}$ converges weakly to $u$; the maps $T_n$ are as in \eqref{NiceTranspMaps}. 
We recall that the statement ``$u_n\circ T_n$ converges weakly to $u$ (in $L^1(\nu)$)'', means that for every $f \in L^\infty(\nu)$, it is true that
\[  \lim_{n \rightarrow \infty} \int_{D} u_n \circ T_n(x) f(x) d \nu(x) = \int_{D} u(x) f(x) d \nu(x).  \]
\begin{remark}
We remark that the notion of weak convergence mentioned previously is not the same as the notion of weak convergence for measures. 
See \cite{CalcVarLeoni} for more on weak convergence in $L^1(\nu)$. Although we use \textit{weak convergence} for convergence of functions and convergence of measures, 
there should be no confusion as to what is the meaning we give to weak convergence in every specific context. 
\end{remark}

Our first simple observation concerns the weak limit of the sequence of functions $\left\{ l_n \right\}_{n \in \N}$.

\begin{lemma}
With probability one, $l_n \rightharpoonup \mu$, where $l_n$ is defined in \eqref{labelfunc} and $\mu$ is defined in \eqref{mu}.
\label{LemmaWeakLabels}
\end{lemma}
\begin{proof}
First recall that with probability one, the empirical measures $\bm{\nu}_n$ converge weakly to the probability measure $\bm{\nu}$ (see \cite{Billingsley}). Secondly, we know that with probability one, the maps $\left\{ T_n \right\}_{n \in \N} $ from \eqref{NiceTranspMaps} exist. We work on a set with probability one where both $\bm{\nu}_n \converges{w} \bm{\nu}$ 
and the transportation maps $T_n$ from $\eqref{NiceTranspMaps}$ exist.

Now, because $|l_n| \leq 1$, by the Dunford-Pettis theorem (see for example \cite{CalcVarLeoni}) the sequence $\left\{ l_n \circ T_n \right\}_{n \in \N}$ is weakly sequentially pre-compact, that is, every subsequence of $\left\{ l_n \circ T_n \right\}$ has a further subsequence which converges weakly. 
Because of this, we may without the loss of generality assume that the sequence $\left\{ l_n \circ T_n \right\}_{n \in \N}$ converges weakly to some $g \in L^1(\nu)$. Our goal is to show that $g= \mu$. 

Let $f \in C_c^\infty(D)$. Then,
\[ \int_{D} l_n \circ T_n(x) f(x) d \nu(x) =  \int_{D} l_n \circ T_n(x)( f(x) - f(T_n(x) ) d \nu(x)  + \int_{D} l_n \circ T_n(x) f(T_n(x)) d \nu(x).\]
Observe that, again because $|l_n| \leq 1$,
\[  \left | \int_{D} l_n \circ T_n(x)( f(x) - f(T_n(x) ) d \nu(x) \right| \leq  \lVert \nabla f  \rVert_{L^\infty(\nu)} \cdot \int_{D} | x - T_n(x)| d \nu(x) \rightarrow 0 , \quad \text{ as } n \rightarrow \infty.\]
Hence,
\[  \int_{D} g(x) f(x) d \nu(x) = \lim_{n \rightarrow \infty}  \int_{D} l_n \circ T_n(x) f(x) d \nu(x) = \lim_{n \rightarrow \infty} \int_{D} l_n \circ T_n(x) f(T_n(x)) d \nu(x).  \]

Using the change of variables formula \eqref{chvar}, and using the fact that $\bm{\nu}_n$ converges to $\bm{\nu}$ weakly, it follows that
\[  \int_{D} g(x) f(x) d \nu(x) = \lim_{n \rightarrow \infty} \frac{1}{n} \sum_{i=1}^{n} f(\x_i)\y_i = \int_{D \times \R} f(x) y d \bm{\nu}(x,y) = \int_{D} \mu(x) f(x) d \nu(x).\]

Since the above formula is true for every $f \in C_c^\infty(D)$, we conclude that $g= \mu$.
\end{proof}

We now determine the ``strong'' limit of the functions $l_n$. Indeed, we show that the functions $l_n$ converge towards the measure $\bm{\nu}$ in the completion of $TL^1(D)$.
In particular, this shows that $l_n$ does not converge to any function $u \in L^1(\nu)$ in the $TL^1$-sense.

\begin{lemma}
With probability one, 
\[ (\nu_n ,  l_n) \converges{d_1} \bm{\nu}, \text{ as } n \rightarrow \infty.\] 
In the above we should interpret $(\nu_n, l_n)$ as a measure in $\mathcal{P}_1(\overline{D} \times \R)$ according to the identification \eqref{embedding} and $d_1$ is the earth mover's distance in $\mathcal{P}_1(\overline{D} \times \R)$.
\label{lnNoConv}
\end{lemma}
\begin{proof}
The result follows from the following simple observations. First, $(Id \times l_n)_{\sharp} \nu_n $ is nothing but $\bm{\nu}_n$. On the other hand, with probability one $\bm{\nu}_n \converges{w} \bm{\nu}$. 
Finally, since the measures $\left\{ \bm{\nu}_n \right\}_{n \in \N}$ have support contained in $\overline{D}\times [0,1]$ (a bounded subset of $\R^d \times \R$), 
we conclude that they have uniformly integrable first moments, and hence $\bm{\nu}_{n} \converges{w} \bm{\nu}$, implies that $\bm{\nu}_{n} \converges{d_1} \bm{\nu}$ (see Chapter 7 in \cite{ambrosio2008gradient}).
\end{proof}

The next observation that we will use in the remainder, concerns the continuity of the risk functionals $R_n$ in the $TL^1$-sense. 

\begin{proposition}[Continuity of risk functional in the $TL^1$-sense]
With probability one the following statement holds: Let $\left\{ u_n \right\}_{n \in \N}$ be a sequence of $[0,1]$-valued functions, with $u_n \in L^1(\nu_n)$. If $u_n  \converges{TL^1} u$ as $n \rightarrow \infty$, then
\[  \lim_{n \rightarrow \infty} R_n(u_n) = R(u). \]
\label{contRisk}
\end{proposition}
\begin{proof}
Because $u_n $ takes values in $[0,1]$ and $l_n$ takes values in $\left\{0,1 \right\}$, we can write
\[ R_n(u_n) = \int_{D} u_n(1-l_n) d \nu_n + \int_{D}(1- u_n) l_n d \nu_n  = \int_{D} u_n d \nu_n + \int_{D}(1 - 2 u_n ) l_n d \nu_n.    \]
Hence,
\[ \lim_{n \rightarrow \infty} R_n(u_n) = \lim_{n \rightarrow \infty} \int_{D} u_n d \nu_n + \lim_{n \rightarrow \infty} \int_{D}(1 - 2 u_n ) l_n d \nu_n = \int_{D} u d \nu + \int_{D} (1 - 2 u) \mu d\nu,   \]
noticing that in the last equality we used the fact that $u_n \converges{TL^1} u$,  $l_n \rightharpoonup \mu$, $|l_n| \leq 1$, $|u| \leq 1$, and Lemma \ref{LemmaWeakLabels}. Finally, observe that the function $u$ must take values in $[0,1]$ and thus the last expression in the above formula can be 
rewritten as $R(u)$. This concludes the proof.

\end{proof}

To finish this section, we present the main results from \cite{GarciaTrillos2015} which state that under the same assumptions on $\left\{ \veps_n \right\}_{n \in \N}$ in Theorem \ref{mainTheorem}, the functional $\sigma_\eta TV$ is the 
$\Gamma$-limit of the functionals $GTV_{n, \veps_n}$ in the $TL^1$-sense. This result will be useful when proving Theorem \ref{mainTheorem} in the regime $\lambda_n \rightarrow \lambda \in (0,\infty]$.

\begin{theorem}[Theorem 1.1, Theorem 1.2 and Corollary 1.3 in \cite{GarciaTrillos2015}]
Let the domain $D$,  measure $\nu$,  kernel $\eta$,  sequence $\{\veps_n\}_{n \in \N}$,  sample points $\{\x_i\}_{i \in \N}$, be as in the statement of Theorem \ref{mainTheorem}. 
Then, with probability one all of the following statements hold simultaneously:
\begin{itemize}
\item \textbf{Liminf inequality:} For every function $u \in L^1(\nu) $ and for every sequence $\left\{ u_n \right\}_{n \in \N}$ with $u_n \converges{TL^1} u$, we have that
\[ \sigma_\eta TV(u) \leq \liminf_{n \rightarrow \infty} GTV_{n,\veps_n}(u_n) .\]
\item \textbf{Limsup Inequality:} For every function $u \in L^1(\nu)$ there exists a sequence $\left\{ u_n \right\}_{n \in \N}$ with $u_n \converges{TL^1} u$, such that
\[  \limsup_{n \rightarrow \infty} GTV_{n,\veps_n}(u_n) \leq \sigma_\eta TV(u) .\]
\item \textbf{Compactness:} Every sequence $\left\{ u_n \right\}_{n \in \N}$ satisfying
\[  \sup_{n \in \N} GTV_{n, \veps_n}(u_n) < +\infty,  \]
is pre-compact in $TL^1$.
\end{itemize}
Moreover, if $u \in L^1(\nu)$ takes only values in $\{0,1 \}$, then in the limsup inequality above, one may choose the functions $u_n \in L^1(\nu_n)$ to take values in $\{0,1 \}$ as well.
\label{ContTV}
\end{theorem}

\section{Proof of Theorem \ref{mainTheorem}}
\label{SecMainTheorem}

\subsection{Overfitting regime $\lambda_n \ll \veps_n $}
\label{OverRegime} 

To prove Theorem \eqref{mainTheorem} in the regime $\lambda_n \ll \veps_n$, we use standard tools from convex analysis. The idea is simply to find the optimality conditions for $u_n^*$.

First, let us write $R_{n, \lambda_n}(u_n)$ as 
\[  \frac{\lambda_n}{  \veps_n n^2} J_{n}(u_n) + \frac{1}{n}\sum_{i=1}^{n} \lvert u_n(\x_i) - l_n(\x_i)  \rvert ,      \] 
where
\[ J_n(u_n):= \sum_{i,j} \eta_{\veps_n}(\x_i- \x_j) \left| u_n(\x_i) - u_n(\x_j)  \right| . \]
In what follows we identify functions $ f \in L^1(\nu_n)$ with vectors in $\R^n$. Namely, a function $f \in L^1(\nu_n)$ is identified with the vector $(f(\x_1), \dots, f(\x_n))$. 
From the minimality of $u_n^*$, we must have
\[  0 \in  \frac{\lambda_n}{ \veps_n  n^2}  \partial J_n(u_n^*) + \frac{1}{n} \partial  \left( \sum_{i=1}^{n} \lvert u_n^*(\x_i) - l_n(\x_i)  \rvert \right) \]
where the $\partial$ symbol denotes sub-gradient. The previous expression implies that there exists $w \in \R^n$ such that:
\begin{equation}\label{auxOver0}
w_i \in 
\begin{cases}
\{ 1 \} & \text{ if } u_n^*(\x_i) > \y_i \\
[-1,1] & \text{ if } u_n^*(\x_i) = \y_i \\
\{ -1 \} & \text{ if } u_n^*(\x_i) < \y_i
\end{cases} 
%\label{auxOver2}
\end{equation}
for every $i =1,\dots, n$; and such that
\[ -\frac{n \veps_n w}{\lambda_n} \in \partial J_n(u_n^*) .\]

The Fenchel dual of $J_n$ is defined by
\[  J_n^*( f ) := \sup_{ g  \in \R^n } \left\{ \sum_{i=1}^{n}g_i f_i - J_n(g)  \right\}.\]
A straightforward consequence of this definition and the fact that $-\frac{n \veps_n w}{\lambda_n} \in \partial J_n(u_n^*)$ is that
\begin{equation}
  u_n^* \in \partial J_n^* \left( -\frac{n \veps_n w}{\lambda_n } \right).
  \label{auxOver} 
\end{equation}

Now, from the fact that $J_n$ is 1-homogeneous (as can be checked easily), it follows that $J_n^*$ has the form:
\begin{equation}
J_n^*(f)= \begin{cases}
0 & \text{ if  } f \in \mathcal{C}_n\\
\infty & \text{ if } f \not \in \mathcal{C}_n, 
\end{cases}
\label{auxOver2}
\end{equation}
where $\mathcal{C}_n$ is a closed, convex subset of $\R^n$. In this case we can give an explicit characterization of $\mathcal{C}_n$ using the following divergence operator. Given $p \in \R^{n^2} $, we define $\divergence(p) \in \R^n$ by:
\[ \divergence(p)_i := \sum_{j=1}^{n}\eta_{\veps_n}(\x_i - \x_j)( p_{ji} - p_{ij}   ), \quad i \in \left\{1, \dots, n \right\}. \]  
By reordering sums, one obtains an analog of the divergence theorem, namely
\[
\sum_{i = 1}^n v_i \divergence(p)_i = \sum_{i,j=1 \dots n} \eta_{\veps_n}(\x_i - \x_j) p_{ij} (v_j - v_i).
\]
This readily implies that
\[
J_n(f) = \sup \left\{ \sum f_i r_i : r_i = \divergence(p)_i, |p_{ij}| \leq 1\right\}.
\]
Since $(J_n^*)^* = J_n$, we have that $J(f) = \sup_{r \in \mathcal{C}_n} \sum_{i=1}^n f_i r_i$, and thus we find that
\[ \mathcal{C}_n = \left\{ \divergence(p) \: : \: p \in \R^{n^2} \text{ s.t } \lvert p_{ij}\rvert \leq 1, \text{ } \forall i,j \right\}.\]

From \eqref{auxOver} we know in particular that $\partial J_n^* \left( -\frac{n \veps_n w}{\lambda_n } \right) \not = \emptyset$. On the other hand, from \eqref{auxOver2}, we conclude that $-\frac{n \veps_n w}{\lambda_n } \in \mathcal{C}_n$.  In turn, this implies that there exists $p \in \R^{n^2}$ with $|p_{ij}|\leq 1$ for all $i,j$ and such that:
\[  \divergence(p)=  -\frac{n \veps_n w}{\lambda_n }. \]
In particular, for all $i =1, \dots, n$,
\begin{align}
\begin{split}
\lvert w_i \rvert  & \leq \frac{2 \lambda_n}{\veps_n} \frac{1}{n} \sum_{j=1}^{n}\eta_{\veps_n}(\x_i - \x_j) 
\\ & =  \frac{2 \lambda_n}{\veps_n } \int_{D} \eta_{\veps_n} \left( T_n(x) - \x_i \right) d \nu(x) 
\end{split}
\label{auxOver5}
\end{align}

Let us introduce the kernel $\hat{\eta}: [0, \infty) \rightarrow \R$ given by
\begin{equation*}
 \hat{\eta}(r) :=\begin{cases}  \eta(0) & \text{ if } r \in [0,1]  \\ \eta(r -1) & \text{ if } r >1.\end{cases}
\label{hateta}
\end{equation*}
Notice that from \eqref{NiceTranspMaps} and the assumptions on $\veps_n$ ( i.e. \eqref{vepsn}), it follows that for all large enough $n$, $ \frac{\lVert Id- T_n \rVert_{L^\infty(\nu)}}{\veps_n} \leq 1$. 
In particular, for all large enough $n$, it follows from the definition of $\hat{\eta}$ that for all $i=1, \dots, n$ and for all $x \in D$
\[ \eta_{\veps_n}\left(T_n(x) -\x_i\right) \leq \hat{\eta}_{\veps_n}\left(x -\x_i\right).\]
Going back to \eqref{auxOver5}, this shows that for every $i=1,\dots,n$ 
\[ \lvert w_i \rvert \leq  \frac{2 \lambda_n}{\veps_n } \int_{D} \hat{\eta}_{\veps_n} \left( x - \x_i \right) d \nu(x) \leq \frac{2M \lambda_n}{\veps_n} \int_{\R^d}\hat{\eta}(x) dx,   \]
where we have used \eqref{BoundsRho}. Because of this, and from the fact that $\frac{\lambda_n}{\veps_n} \rightarrow 0$, we conclude that if $n$ is large enough, $|w_i| < 1$ for all $i=1, \dots, n$. 

Thus by \eqref{auxOver0}, for $n$ sufficiently large we have that
\[ u_n^*(\x_i) = \y_i , \quad \forall i= 1, \dots,n. \]
In short, this means that for all large enough $n$, $u_n^* = l_n $. Since $l_n$ does not converge in $TL^1$ to a function as $n \rightarrow \infty$ (see Lemma \ref{lnNoConv}), we conclude that the same is true for the sequence $\left\{ u_n^* \right\}_{n \in \N}$. 
% Note that the above computations show something stronger than what was claimed in Theorem \ref{mainTheorem}, namely, we showed that $u_n^* = l_n $ for large enough $n$. 
% In other words, assymptotically, the regularizer is not doing anything.
%  

%\begin{theorem} 

%With probability one the following statement holds: If a sequence $\left\{ u_n^* \right\}$ of minimizers of $R_{n, \lambda_n}$ (alternatively almost minimizers), satisfies
%\[ u_n^* \converges{TL^1} u \]
%for some function $u \in L^1(\nu)$, then $u=u_B$.
%\end{theorem}
%\begin{proof}
%The proof uses ideas from the $\Gamma$-convergence literature. Let us consider the following  ...
%\end{proof}

%
%\begin{remark}
%Observe that in the above theorem, the only assumption on $\lambda_n$ is that it converges to zero as $n \rightarrow \infty$. Also, the theorem does not say that $u_n^*$ converges to $u_B$. 
%It simply states that if asequence of minimizers converges to some function, then that function has to be the Bayes classifier. As we will see, depending on the scaling of $\lambda_n$ with respect to $n$, 
%the sequence $\left\{ u_n^* \right\}$ will actually converge in $TL^1$ or not.  
%\end{remark}

%Let us first establish that if $\lambda_n \ll \veps_n$, then there is no hope to converge in $TL^1$. This is the content of our first Theorem.

%\begin{theorem}
%Suppose that $\lambda_n$ is such that ...  Then, 
%\[ u_n^* := \argmin_{u \in L^1(\nu_n)} R_{n , \lambda_n}(u)  \]
%is not precompact in $TL^1$. 
%\end{theorem}

\subsection{Underfitting regime: $ \lambda_n \rightarrow \lambda \in (0,\infty]$.  }
\label{UnderRegime}

Now we establish Theorem \ref{mainTheorem} in the underfitting regime $\lambda_n \rightarrow \lambda \in (0,\infty]$. 
The main tool we have at hand to study this regime is Theorem \ref{ContTV}. In particular, we will use the compactness result from Theorem \eqref{ContTV}.

First of all, notice that for every $n \in \N$
\begin{equation}
   \lambda_n GTV_{n, \veps_n}(u_n^*)   \leq    R_{n, \lambda_n}(u_n^*) \leq \inf_{ y \in \R}  \frac{1}{n} \sum_{i=1}^{n} \left| y- \y_i \right| \leq 1,         
   \label{UnderCompact}
\end{equation}
and so in particular, $GTV_{n, \veps_n}(u_n^*) \leq \frac{1}{\lambda_n} $. Since $\lambda_{n} \rightarrow \lambda \in (0,\infty]$, we conclude that 
\[  \sup_{n \in \N} GTV_{n, \veps_n}(u_n^*) < + \infty . \]

From the compactness statement in Theorem \ref{ContTV}, we deduce that $ \left\{ u_n^* \right\}_{n \in \N}$ is pre-compact in $TL^1$.

\textbf{Case 1:} Let us assume first that $\lambda_n \rightarrow \infty$. In this case, from \eqref{UnderCompact}, we actually deduce that, 

\begin{equation}
 \lim_{n \rightarrow \infty} GTV_{n, \veps_n} (u_n^*) =0     
\label{StrongUnder} 
\end{equation}

Now, by the pre-compactness of $\left\{ u_n^* \right\}_{n \in \N}$, we know that up to subsequence (that we do not relabel), $\left\{ u_n^* \right\}_{n \in \N}$ converges in the $TL^1$-sense towards some $u \in  L^1(\nu)$. From the lower semi-continuity of the graph total variation (i.e. the liminf inequality in Theorem \ref{ContTV}) and from \eqref{StrongUnder}, we deduce that $TV(u)=0$. The connectedness of the domain $D$ implies that $u$ is constant on $D$. That is, $u\equiv a$ for some $a \in \R$. 
Because, $\bm{\nu}_n \converges{w} \bm{\nu}$, we know that for every $b \in \R$,
\[ \lim_{n \rightarrow \infty} \frac{1}{n} \sum_{i=1}^{n} \lvert b - \y_i \rvert = \int_{D\times \R} \lvert b - y \rvert d \bm{\nu}(x,y) = R(b).    \]
On the other hand, for a given $b \in \R$,
\[ \frac{1}{n} \sum_{i=1}^{n} \lvert u_n^*(\x_i) - \y_i \rvert  \leq R_{n, \lambda_n}(u_n^*) \leq \frac{1}{n} \sum_{i=1}^{n} \lvert b- \y_i \rvert.\]
Additionally, from $u_n^* \converges{TL^1} a$, it is straightforward to check that
\[ \lim_{n \rightarrow \infty} \frac{1}{n} \sum_{i=1}^{n} \lvert u_n^*(\x_i) - \y_i \rvert  =  \lim_{n \rightarrow \infty} \frac{1}{n} \sum_{i=1}^{n} \lvert a - \y_i \rvert = R(a). \]
From the previous computations we deduce that $R(a) \leq R(b)$ for every $b \in \R$. This shows that $a=  u^\infty$ where $u^\infty$ is defined in \eqref{uinfty}. We have just shown that for every subsequence of $\left\{ u_n^* \right\}_{n \in \N}$, there is a further subsequence converging towards $u^\infty$. Thus, the full sequence $\left\{ u_n^* \right\}_{n \in \N}$ converges towards $u^\infty$ in the $TL^1$-sense as we wanted to show.
Finally, from Proposition \ref{contRisk} it follows that $\lim_{n \rightarrow \infty}R_n(u_n^*) = R(u^\infty) = \min_{y \in \R} R(y)$. 

\textbf{Case 2:} Let us now assume that $\lambda_n \rightarrow \lambda \in (0,\infty)$. From Proposition \ref{contRisk} and from the $\Gamma$-convergence of $GTV_{n, \veps_n}$ towards $\sigma_\eta TV$ (Theorem \ref{ContTV}) it is immediate that $R_{n , \lambda_n} \converges{\Gamma} R_\lambda$ as $n \rightarrow \infty$ in 
the $TL^1$-sense. Indeed, in \cite{DalMaso} the $\Gamma$-convergence of continuous perturbations of a $\Gamma$-converging sequence is considered: in our case we are perturbing the functionals $\lambda_n GTV_{n, \veps_n}$ with $R_n$. From the fact that $R_{n , \lambda_n} \converges{\Gamma} R_\lambda$ in the $TL^1$-sense and the fact that $\left\{ u_n^* \right\}_{n \in \N}$ is pre-compact in $TL^1$, 
it follows that every subsequence of $u_n^*$ has a further subsequence converging to a minimizer of $R_\lambda$. From the properties of $\Gamma$-convergence (see \cite{DalMaso}), 
it also follows that $\lim_{n \rightarrow \infty}R_{n, \lambda_n}(u_n^*) = \min_{u \in L^1(\nu)}R_{\lambda}(u).  $

\subsection{Regime  $ \veps_n \ll \lambda_n \ll 1$}
\label{GoodRegime}

%Before we give an outline of the proof of Theorem \ref{mainTheorem} in the regime $\veps_n \ll \lambda_n \ll 1$, we introduce some terminology that we will use in the rest of this section. 

%
%Given a sequence $\left\{ u_n \right\}_{n \in \N}$ with $u_n \in L^1(\nu_n)$, we say that $\left\{ u_n \right\}_{n \in \N}$ \textit{converges weakly} to $u \in L^1(\nu)$ (and denote this convergence by $u_n \rightharpoonup u$) if the sequence of functions $\left\{ u_n \circ T_n \right\}_{n \in \N}$ converges weakly to $u$; the maps $T_n$ are as in \eqref{}. 
%We recall that the statement ``$u_n\circ T_n$ converges weakly to $u$ (in $L^1(\nu)$)'', means that for every $f \in L^\infty(\nu)$, it is true that
%\[  \lim_{n \rightarrow \infty} \int_{D} u_n \circ T_n(x) f(x) d \nu(x) = \int_{D} u(x) f(x) d \nu(x).  \]
%\begin{remark}
%We remark that the notion of weak convergence mentioned previously is not the same as the notion of weak convergence for measures. See \cite{} for more on weak convergence in $L^1(\nu)$.
%\end{remark}

The idea of the proof of Theorem \ref{mainTheorem} in the regime $\veps_n \ll \lambda_n \ll 1$ is as follows. We establish that if the sequence $\left\{ u_n^*  \right\}_{n \in \N}$ converges weakly to some function $u \in L^1(\nu)$ (recall the definition of weak convergence given at the beginning of Subsection \ref{AuxResul}), then the convergence also happens in the $TL^1$-sense. 
Then, we establish that if $u_n^*$ converges weakly to some function $u \in L^1(\nu)$, and additionally
\begin{equation}
 \lim_{n \rightarrow \infty} \int_{D} u_n^*(x) l_n(x) d \nu_n(x) = \int_{D}  u(x) \mu(x) d\nu(x), 
 \label{ConditionConvergence}
\end{equation}
then $u$ has to be equal to the Bayes classifier $u_B$. So in order to establish that $u_n^* \converges{TL^1} u_B$, it will be enough to show that $u_n^*$ converges weakly to some $u$ and that \eqref{ConditionConvergence} is satisfied. 

Now, since $D$ is a bounded set in $\R^d$ and since all the functions $u_n^*\circ T_n$ are uniformly bounded in $L^\infty(\nu)$, it follows from Dunford-Pettis theorem (see for example \cite{CalcVarLeoni}), that 
the sequence $\left\{ u_n^* \circ T_n \right\}_{n \in \N}$ is weakly sequentially pre-compact, that is, every subsequence of $\left\{ u_n^* \circ T_n \right\}_{n \in \N}$ has a further subsequence which converges weakly. 
Because of this, we may without the loss of generality assume that the sequence $\left\{ u_n^* \right\}_{n \in \N}$ converges weakly to some $u \in L^1(\nu)$. Hence the task is to show that \eqref{ConditionConvergence} holds in the regime $\veps_n \ll \lambda_n \ll 1$.

To establish \eqref{ConditionConvergence}, we heuristically observe that the oscillations of the functions $u_n^*$ happen at a scale larger than $\veps_n$, whereas 
the oscillations of $l_n$ happen at a scale smaller than $\veps_n$; the statement regarding the oscillations of the functions $u_n^*$ is related to the fact that the energies $ \lambda_n GTV_{n,\veps_n}(u_n^*) $ are uniformly bounded and the fact that  $\veps_n \ll \lambda_n \ll 1$, on the other hand, 
the statement regarding the oscillations of the functions $l_n$ is a direct consequence of concentration inequalities. Heuristically, we may think of the function $u_n^*$ as constant on balls of radius $\veps_n$,
whereas we may view the functions $l_n$ as rapidly oscillating on those same neighborhoods; because of this, when integrating over such neighborhoods, the functions $l_n$ behave like their weak limit (i.e. the function $\mu$, see Lemma \ref{LemmaWeakLabels}). 

There are certain connections between the ideas in the proofs here and the theory of fractional Sobolev spaces. In particular, the consistency regime $ \lambda_n GTV_{n,\veps_n}(u_n^*) $ has scaling similar to a fractional Sobolev seminorm. Hence the argument that we use of approximating $u_n^*$ with functions that are constant on a length scale $\veps_n$ is not unlike the argument used to prove the compactness of fractional Sobolev spaces, see e.g. the proof of Theorem 7.1 in \cite{DiNezza2012}.
 
With this road-map in mind let us start making the previous statements precise.
\begin{lemma}
With probability one the following statement holds: Let $ \left\{ u_n \right\}_{n \in \N} $ be a sequence of $[0,1]$-valued functions, with $u_n \in L^1(\nu_n)$, and such that $u_n \rightharpoonup u$ for some function $u \in L^1(\nu)$ taking only the values $0$ and $1$. Then, $u_n \converges{TL^1} u$ as $n \rightarrow \infty$.
\label{LemmaGoodRegime1}
\end{lemma}

\begin{proof}
We may work on a set of probability one, where all the statements in Theorem \ref{ContTV} hold. Let the sequence $\left\{ u_n \right\}_{n\in \N}$ and the function $u$ satisfy the hypothesis in the statement of the lemma.  We know that there exists a sequence $\left\{ w_n \right\}_{n \in \N}$ with 
 \[ w_n \converges{TL^1} u \]
and such that $w_n \in \left\{0,1 \right\} $. The existence of such sequence of functions follows in particular from the last statement in Theorem \ref{ContTV}. 
Then, from the fact that $w_n\in \{0,1 \}$ and $u_n \in [0,1]$, it is straightforward to see that
\[ \int_D | w_n - u_n| d \nu_n = \int_D u_n d \nu_n +  \int_D (1-2 u_n) w_n d \nu_n.  \]

Using the fact that $w_n \converges{TL^1} u$ (strong convergence), $u_n \rightharpoonup  u$ (weak convergence), and that $u_n,w_n$ are uniformly bounded, we deduce that
\[  \lim_{n \rightarrow \infty}  \int_D |w_n- u_n | d \nu_n  = \lim_{n \rightarrow \infty} \int_D u_n d \nu_n + \lim_{n \rightarrow \infty}  \int_D  (1-2 u_n) w_n  d \nu_n = \int_D u d \nu + \int_D (1- 2 u) u d \nu =0;\]
note that in the last equality we have used the fact that $u^2 = u$. Given that $w_n \converges{TL^1} u$, we conclude that $u_n \converges{TL^1} u$ as well.
\end{proof}

\begin{lemma}
With probability one the following statement holds: if a sequence of minimizers $\left\{u_n^* \right\}_{n \in \N}$ of the energies $R_{n,\lambda_n}$ satisfies $u_n^* \rightharpoonup u$ for some function $u \in L^1(\nu)$ and in addition condition \eqref{ConditionConvergence} holds, then $u=u_B$.
\label{LemmaGoodRegime2}
\end{lemma}
\begin{proof}
We know that with probability one, for the function $u_B$, there exists a sequence $\left\{ u_n \right\}_{n \in \N}$ of $\left\{0,1 \right\}$-valued functions with $u_n \in L^1(\nu_n)$, such that $u_n \converges{TL^1} u_B$ as $n \rightarrow \infty$ and such that 
$\limsup_{n \rightarrow \infty} GTV_{n,\veps_n}(u_n) \leq \sigma_\eta TV(u_B) < +\infty$; this follows from the last statement in Theorem \ref{ContTV} and the fact that we assumed that $u_B$ has finite total variation.
From this, the fact that $\lambda _n \rightarrow 0$ and Lemma \ref{contRisk}, we deduce that
\[  \limsup_{n \rightarrow \infty} \lambda_n GTV_{n,\veps_n}(u_n)  + R_n(u_n) =   R(u_B).     \]

On the other hand, since $u_n^*$ minimizes $R_{n,\lambda_n}$, we conclude that
\begin{equation}
 \limsup_{n \rightarrow \infty} R_{n , \lambda_n}(u_n^*) \leq  \limsup_{n \rightarrow \infty} \left( \lambda_n GTV_{n,\veps_n}(u_n)  + R_n(u_n) \right)=   R(u_B).
\label{RiskConvergence}
\end{equation}

Now, given that $u_n^*$ minimizes $R_{n, \lambda_n}$ it is clear that $u_n^*$ takes values in $[0,1]$ only, and thus we can write 
\[ R_n(u_n^*) =  \int_D l_n d \nu_n + \int_D (1-2 l_n ) u_n^* d\nu_n. \]
From \eqref{ConditionConvergence}, Lemma \ref{LemmaWeakLabels}, and the fact that $u_n^* \rightharpoonup u$, we deduce that
\[  \lim_{n \rightarrow \infty} R_n(u_n^*) = \lim_{n \rightarrow \infty} \left( \int_D l_n d \nu_n + \int_D (1-2 l_n ) u_n^* d\nu_n\right) = \int_D \mu d \nu  + \int_D (1- 2 \mu) u d\nu = R(u).      \]
where the last equality follows from the fact that $u$ must take values in $[0,1]$. Since we clearly have $R_n(u_n^*) \leq R_{n , \lambda_n}(u_n^*)$ for every $n$, 
we deduce from the above equality and \eqref{RiskConvergence}, that
\[  R(u) \leq R(u_B). \]
The fact that $u_B$ is the unique minimizer of $R$ implies that $u=u_B$ as we wanted to show.
\end{proof}

In light of Lemma \ref{LemmaGoodRegime1}, Lemma \ref{LemmaGoodRegime2}, the fact that $u_B$ takes values in $\{0,1\}$ and the discussion at the beginning of 
this subsection, to show that $u_n^* \converges{TL^1} u_B $, it remains to show that when $u_n^* \rightharpoonup u$ for some $u \in L^1(\nu)$, \eqref{ConditionConvergence} holds.  The remainder of the section is devoted to this purpose.

Let us consider a sequence $\left\{ \veps_n \right\}_{n \in \N}$ of positive numbers converging to zero satisfying \eqref{vepsn}. For every $n \in \N$ we consider a family of disjoint balls  $B(z_1, \veps_n/4),\dots, B(z_{k_n}, \veps_n/4) $ satisfying the following conditions:
\begin{enumerate}
\item Every $z_i$ belongs to $D$.
\item The family of balls is maximal, in the sense that every ball $B(z, \veps_n/4)$  with $z \in D$, intersects at least one of the balls $B(z_i, \veps_n /4)$. 
\end{enumerate}
We let $S_n:= \left\{z_1,\dots, z_{k_n} \right\}$. By the maximality property of the family of balls $\left\{B(z,\veps_n/4) \right\}_{z \in S_n}$, we see that $\left\{B(z,\veps_n/2) \right\}_{z \in S_n}$ covers $D$. Moreover, we claim that there is a constant $C>0$ such that,
\begin{equation}
 |S_n | \leq \frac{C}{\veps_n^d}. 
\label{CardSn}
\end{equation}
 To see this, 
% we first observe that for every $x \in D$, $\nu(B(x, \veps_n/4)) \geq c  \veps_n^d$, where $c>0$ does not depend on $x$ nor on $n$. Indeed, notice that if $x\in D$ was $\veps_n/4$ away from the boundary of $D$, then 
%\[  \nu(B(x,\veps_n/4)) = \int_{B(x, \veps_n/4) \cap D} \rho(x) dx   =  \int_{B(x, \veps_n/4) \cap D} \rho(x) dx \geq  \frac{\omega_d}{m4^d}\veps_n^d,  \]
%where $\omega_d$ is the volume of the unit ball in $\R^d$ and where we have used the fact that  $\rho$ satisfies \eqref{BoundsRho}. The problem with the previous estimate is that it does not work for points $x\in D$ that are within distance $\veps_n/4$ of the boundary of $D$. 
we may use the regularity assumption on the boundary of $D$
as follows. From the fact that $D$ is an open and bounded set with Lipschitz boundary it follows (see \cite{Gris}, Theorem 1.2.2.2) that there exists a cone $\mathcal{C} \subseteq \R^d$ with non-empty interior and vertex at the origin, a family of rotations $\left\{ R_{x} \right\}_{x \in D}$ and a number $1>\zeta>0$ such that for every $x \in D,$
$$  x + R_x(\mathcal{C} \cap B(0,\zeta))  \subseteq D. $$
Thus, 
\[  \nu(B(x,\veps_n/4)) = \int_{B(x, \veps_n/4) \cap D} \rho(x) dx \geq  \int_{x + \frac{\veps_n}{4}(R_x(\mathcal{C} \cap B(0,\zeta))) } \rho(x) dx  \geq \frac{|\mathcal{C} \cap B(0,\zeta)|}{m4^d} \veps_n^d,\]
where $|\mathcal{C} \cap B(0,\zeta)|$ denotes the volume of $\mathcal{C} \cap B(0,\zeta)$. The bottom line is that there exists a constant $c>0$ such that for every $x \in D$ we have:
\begin{equation}
\nu(B(x,\veps_n/4)) \geq c \veps_n^d.
\label{sizeballs}
\end{equation}
The inequality in \eqref{CardSn} follows now immediately from 
\[  c |S_n| \cdot  \veps_n^d \leq  \sum_{z \in S_n} \nu(B(z, \veps_n /4)  )  = \nu( \cup_{z \in S_n} B(z, \veps_n /4)  )  \leq   \nu(D)=1 . \] 
%}
%{\blue By the boundedness of $D$, it is straightforward to demonstrate that
%\begin{equation}
% |S_n | \leq \frac{C}{\veps_n^d}. 
%\label{CardSn}
%\end{equation}
%Furthermore, by the Lipschitz condition on $D$ and \eqref{BoundsRho}, for any $x \in D$ we have that
%\begin{equation}
%\nu(B(x,\veps_n/4)) \geq c \veps_n^d .
%\label{sizeballs}
%\end{equation}
%}

Let $\left\{\psi_{z} \right\}_{z \in S_n}$ be a smooth partition of unity subordinated to the open covering $\left\{B(z,\veps_n) \right\}_{z \in S_n} $.  We remark that the functions $\psi_z$ can be chosen to satisfy
\begin{equation}
\lVert \nabla  \psi_z \rVert_{L^\infty(\R^d)} \leq \frac{C}{\veps_n}  ,
\label{UniformBoundGradients}
\end{equation}
where $C>0$ is a constant independent of $n$ or $z \in S_n$ (see e.g. the construction in Theorem C.21 in \cite{Leoni}). 

The following lemma is a an important first step in proving \eqref{ConditionConvergence}. The proof uses concentration inequalities to control oscillations on a small length scale.

\begin{lemma}
Let $(\x_1, \y_1), ( \x_2, \y_2), \dots, (\x_n, \y_n), \dots$ be i.i.d. samples from $\bm{\nu}$.  Assume that $\left\{ \veps_n \right\}_{n \in \N}$ is a sequence of positive numbers satisfying:
\[    \frac{(\log(n))^{1/d}}{n^{1/d}} \ll \veps_n \ll 1 .\]
Then, with probability one 
\[  \lim_{n \rightarrow \infty} \sum_{z \in S_n} \left| \frac{1}{n}\sum_{i=1}^{n}\left( \mu(\x_i) - \y_i \right) \cdot \psi_z(\x_i)   \right|   =0.\]
\label{LemmaAuxWeak}
\end{lemma}

\begin{proof}

Fix $\beta>2$. Let $z \in S_n$ and let $N_z:= \# \{ i \in \{ 1, \dots, n \} \: : \: \x_i \in B(z, \veps_n) \}$. In the event where the transportation map $T_n$ from \eqref{NiceTranspMaps} exists (this event occurs with probability at least $1-1/ n^{\beta}$), we have that
\[ \lVert T_n - Id \rVert_{L^\infty(\nu)} \leq C_\beta \frac{\log(n)^{p_d}}{n^{1/2}} \ll \veps_n,\]
and from this, it follows that
\[  \bigcup_{\x_i \in B(z, \veps_n)} T^{-1}_n( \left\{ \x_i  \right\}) \subseteq B(z,2 \veps_n).  \]
We conclude that with probability at least  $1- 1/n^{\beta}$,
\begin{equation}
\frac{N_z}{n} \leq \nu(B(z, 2\veps_n)) \leq  M C_d \veps_n^d,
\label{Nzsize} 
\end{equation}
where $M$ is as in \eqref{BoundsRho} and $C_d$ is a constant only depending on dimension.

On the other hand, conditioned on $\x_i=x_i $ for $i=1,\dots,n$, the variables $ \left\{  \y_i  \cdot \psi_z(\x_i) \right\}_{i=1, \dots, n}$
are conditionally independent and have conditional distribution:
 \begin{equation*}
 \y_i\psi_z(\x_i) =
 \begin{cases}
  \psi_z(\x_i) & \text{with prob. }  \mu(\x_i)
  \\ 0 & \text{with prob. } 1- \mu(\x_i). 
 \end{cases}
 \end{equation*}
 Hence by Hoeffding's inequality, for every $t>0$, we have
\begin{align}
\begin{split}
\Prob \left( \left| \frac{1}{n} \sum_{i=1}^{n} ( \mu(\x_i)- \y_i ) \cdot \psi_z(\x_i) \right| > t \quad  |\quad \x_i = x_i,\quad  \forall i \in \left\{ 1,\dots,n \right\}\right) & \leq 2 \exp \left( -\frac{2 n^2 t^2}{\sum_{i=1}^{n} (\psi_z(x_i))^2 }  \right) 
\\& \leq  2 \exp \left( -\frac{2 n^2 t^2}{N_z }\right),
\end{split}
\label{Hoeffdingsineq}
\end{align}
where the second inequality follows from the fact that $\psi_z$ is always less than $1$.  

From \eqref{Nzsize} and \eqref{Hoeffdingsineq} (taking $t= \sqrt{\frac{\beta MC_d \log(n) \veps_n^d }{ n}}$ ), we deduce that with probability at least $1-2/n^{\beta} $ we have  

\begin{equation}  
\left| \frac{1}{n} \sum_{i=1}^{n} ( \mu(\x_i)- \y_i ) \cdot \psi_z(\x_i) \right|  \leq \sqrt{\frac{MC_d \beta \log(n) \veps_n^d }{ n}} . 
\label{auxHoeffding}
\end{equation}

In the previous estimate we used $z\in S_n$ fixed. Now, using a union bound (where the index set is $S_n$) we deduce from \eqref{CardSn} that with probability at least $1 - \frac{2C}{n^{\beta} \veps_n^d}$, \eqref{auxHoeffding} holds for every $z \in S_n$.

Therefore, with probability at least $1 - \frac{2C}{n^{\beta} \veps_n^d}$  , 
\[  \sum_{z \in S_n} \left| \frac{1}{n} \sum_{i=1}^{n} ( \mu(\x_i)- \y_i ) \cdot \psi_z(x_i) \right|  \leq  \frac{C}{\veps_n^d}\cdot \sqrt{\frac{\beta \log(n) \veps_n^d }{ n}}= C\sqrt{\frac{\beta \log(n)}{ n \veps_n^d}}  \] 

Since $\frac{1}{n^{\beta} \veps_n^d}$ is summable (notice that $\frac{1}{n^{\beta} \veps_n^d} \ll \frac{1}{n^{\beta-1}}$ ) we can use Borel-Cantelli lemma to conclude that with probability one
\[ \lim_{n \rightarrow \infty} \sum_{z \in S_n} \left| \frac{1}{n} \sum_{i=1}^{n} ( \mu(\x_i)- \y_i ) \cdot \psi_z(\x_i) \right| =0,   \]
which is what we wanted to show.

%Now recall that for every $z \in S_n$, $\overline{ u_n^T}(z) \leq 1$. Thus, 

%\[  \left| \sum_{z \in S_n} \overline{ u_n^T}(z) \int_{B(z, \veps_n)}     -  \sum_{z \in S_n} \overline{ u_n^T}(z) \int_{B(z, \veps_n)}        \right| \leq  |S_n| \cdot \sqrt{\frac{\beta \log(n) \veps_n^d }{ n}},  \]
%where $|S_n|$ stands for the cardinality of $S_n$. It is simple to see that $|S_n| \leq \frac{C}{\veps_n^d}$. Hence,

%\[   \left| \sum_{z \in S_n} \overline{ u_n^T}(z) \int_{B(z, \veps_n)}     -  \sum_{z \in S_n} \overline{ u_n^T}(z) \int_{B(z, \veps_n)}        \right| \leq  \sqrt{\frac{\beta \log(n)}{ n \veps_n^d }}.    \]

%Given the assumptions on $\veps_n$, we conclude that the last term converges to zero as $n \rightarrow \infty$.

%
%\nc

\end{proof}

With all the previous lemmas at hand, we are now ready to complete the proof of Theorem \ref{mainTheorem}.

\begin{proof}[Proof of Theorem \ref{mainTheorem}, part (2)] 

Following the arguments at the start of the section we may safely assume that $u_n^* \rightharpoonup u$ for some $u \in L^1(\nu)$. Lemmas \ref{LemmaGoodRegime1} and \ref{LemmaGoodRegime2} then imply that if \eqref{ConditionConvergence} holds, then $u_n^* \converges{TL^1} u_B$, which is the desired result. Hence the remainder of the proof aims to show \eqref{ConditionConvergence}.

First of all observe that 
\begin{equation}
 \sup_{n \in \N} \lambda_n GTV_{n, \veps_n}(u_n^*) \leq 1,
\label{Regularity}
\end{equation}

which follows from the fact that for every $n \in \N$,
\[ \lambda_n GTV_{n, \veps_n}(u_n^*) \leq R_{n, \lambda_n}(u_n^*) \leq R_{n, \lambda_n}(1) = R_n(1) \leq 1.  \]

Consider $u_n^T := u_n^* \circ T_n$, where $T_n$ is the transportation map from \eqref{NiceTranspMaps}. Likewise, define $l^T_n(x) := l_n\circ T_n (x)$. Observe that for almost every $x,w \in D$, we have
\[  \frac{|T_n(x) - T_n(w)|}{\veps_n} \leq  \frac{|x-w|}{\veps_n}  + \frac{2 \lVert Id - T_n \rVert_{L^\infty(\nu)}}{ \veps_n}. \]
Now, given \eqref{NiceTranspMaps} and \eqref{vepsn} we conclude that for all large enough $n$ and for almost every $x,w \in D$ we have
\[ \hat{\eta}\left( \frac{x-w}{\veps_n} \right) \leq  \eta\left( \frac{T_n(x) - T_n(w)}{\veps_n} \right) ,\]
where $\hat{\eta}$ is defined as
\begin{equation}
 \hat{\eta}(r) := \eta(r +1)  \text{ for } r \geq0.
\label{etah2}
\end{equation}

In particular, from \eqref{Regularity} we deduce that
\begin{equation}
    \sup_{n \in \N} \frac{\lambda_n}{\veps_n} \int_{D \times D} \hat{\eta}_{\veps_n}(x-w) |u_n^T(x)  - u_n^T(w)|  d \nu(x) d \nu(w)  < \infty ,
    \label{UniformBoundFractionalSobolev}
\end{equation}
where we have used the change of variables \eqref{chvar} to write integrals with respect to $\nu_n$ as integrals with respect to $\nu$.

Using again the change of variables \eqref{chvar}, we can restate our original goal to be
\begin{equation}
\lim_{n \rightarrow \infty} \int_D u_n^T l^T_n d \nu  = \int_D u \mu d \nu.   
  \label{Goal}
\end{equation}
We show \eqref{Goal} in several steps. 

First, for $z \in S_n$, we consider the average: 
\[  \overline{u^T_n}(z) := \frac{1}{\nu\left( B(z,\veps_n{}) \right)} \int_{B(z,\veps_n{})} u_n^T(w ) d\nu(w).\]

Then, we notice that

\begin{align*}
\begin{split}
& \left \lvert \int_D u_n^T(x) l^T_n(x)  d \nu(x) -   \sum_{z \in S_n}  \overline{u^T_n}(z) \int_{B(z,\veps_n)} l^T_n(x) \psi_z(x) d \nu(x)   \right \rvert 
\\& = \left  \lvert  \sum_{z \in S_n}  \int_{D} u_n^T(x) l^T_n(x) \psi_z(x) d \nu(x) -   \sum_{z \in S_n}   \int_{B(z,\veps_n)} \overline{u^T_n}(z) l^T_n(x) \psi_z(x) d \nu(x)   \right \rvert  
\\&= \left  \lvert  \sum_{z \in S_n}  \int_{B(z,\veps_n)} u_n^T(x) l^T_n(x) \psi_z(x) d \nu(x) -   \sum_{z \in S_n}   \int_{B(z,\veps_n)} \overline{u^T_n}(z) l^T_n(x) \psi_z(x) d \nu(x)  \right  \rvert
 \\  &\leq  \sum_{z \in S_n} \frac{1}{\nu\left( B(z,\veps_n) \right)} \int_{B(z,\veps_n)} \int_{B(z, \veps_n)} \lvert u_n^T(x) - u_n^T(w) \rvert  l^T_n(x) \psi_z(x) d \nu(w) d \nu(x) 
\\ & \leq \frac{C}{\veps_n^d} \int_{D} \int_{B(x,\veps_n)}  \lvert u_n^T(x) - u_n^T(w) \rvert d \nu(w) d \nu(x)
\\& \leq C \int_{D} \int_{B(x,\veps_n)} \hat{\eta}_{\veps_n}(x-w)  \lvert u_n^T(x) - u_n^T(w) \rvert d \nu(w) d \nu(x)
\end{split}
\end{align*}
where in the first equality we have used the fact that the functions $ \{ \psi_z \}_{z \in S_n}$ form a partition of unity; in the second equality we have used the fact that $\psi_z$ is supported in $B(z,\veps_n)$; we have also used the fact that $|l^T_n|$ and $\psi_z$ are bounded above by one and the fact that $ \nu( B(z,\veps_n) )  \geq c \veps_n^d$ (see \eqref{sizeballs}); the last inequality follows from the assumption \eqref{ExtraAssump} and the definition of $\hat{\eta}$ in \eqref{etah2}.

 From \eqref{UniformBoundFractionalSobolev} and the fact that $\frac{\veps_n}{\lambda_n} \rightarrow 0$ (by assumption), we deduce that 
\begin{equation}
\lim_{n \rightarrow \infty} {} \left \lvert  \int_D u_n^T(x) l^T_n(x)  d \nu(x) -   \sum_{z \in S_n}  \overline{u^T_n}(z) \int_{B(z,\veps_n{})} l^T_n(x) \psi_z(x) d \nu(x)   {}   \right\rvert =0 .
\label{auxweak1}
\end{equation}
In a similar fashion we obtain  
\begin{align}
\begin{split}
{} \left \lvert \int_D u_n^T(x) \mu(x)  d \nu(x) -   \sum_{z \in S_n}  \overline{u^T_n}(z) \int_{B(z,\veps_n{})} \mu(x)\psi_z(x) d \nu(x)   {} \right \rvert  
 \\ \leq C \int_{D} \int_{B(x,\veps_n)} \hat{\eta}_{\veps_n}(x-y)  \lvert u_n^T(x) - u_n^T(w) \rvert d \nu(w) d \nu(x),
 \end{split}
\end{align}
and thus
\begin{equation}
\lim_{n \rightarrow \infty} {} \left \lvert \int  u_n^T(x) \mu(x)  d \nu(x) -   \sum_{z \in S_n}  \overline{u^T_n}(z) \int_{B(z,\veps_n{})} \mu(x)\psi_z(x) d \nu(x)  { \rvert} \right \rvert =0 .
\label{auxweak2}
\end{equation}

On the other hand, notice that 
\begin{equation}
\lim_{n \rightarrow \infty } {} \left \lvert \int_D u_n^T(x) \mu(x) d \nu(x)  -   \int_D  u(x) \mu(x) d \nu(x) {} \right \rvert = 0,
\label{auxweak3}
\end{equation}
which follows directly from the fact that $u_n^T$ converges weakly towards $u $. 

From \eqref{auxweak1}, \eqref{auxweak2}, \eqref{auxweak3} and the triangle inequality, it follows that in order to  show \eqref{Goal} it is enough to show that
\[  \lim_{n \rightarrow \infty} \left| \sum_{z \in S_n}  \overline{u^T_n}(z) \int_{B(z,\veps_n{})} \mu(x) \psi_z(x)d \nu(x)  -    \sum_{z \in S_n}  \overline{u^T_n}(z) \int_{B(z,\veps_n{})} l^T_n(x) \psi_z(x) d \nu(x)  \right| =0.\]
However, notice that 
\begin{align*}
\begin{split}
 &\left| \sum_{z \in S_n}  \overline{u^T_n}(z) \int_{B(z,\veps_n{})} \mu(x) \psi_z(x)d \nu(x) - \sum_{z} \overline{u^T_n}(z) \int_{B(z,\veps_n{})} \mu(T_n(x)) \psi_z(x)d \nu(x)  \right| 
 \\ & \leq \sum_{z \in S_n}  \int_{B(z,\veps_n{})} \lvert \mu(x)- \mu(T_n(x)) \rvert  \psi_z(x) d \nu(x) 
 \\&= \sum_{z \in S_n}  \int_{D} \lvert \mu(x)- \mu(T_n(x)) \rvert  \psi_z(x) d \nu(x)
 \\&= \int_{D} \lvert \mu(x)- \mu(T_n(x))  \rvert d \nu(x),
\end{split}
\end{align*}
and this last term goes to zero as $n \rightarrow \infty$; this follows from the fact that $\mu$ is continuous at $\nu$-a.e. $x \in D$ and so $ \lim_{n \rightarrow \infty}\mu(T_n(x)) = \mu(x) $ for $\nu$-a.e. $x\in D$, and by the dominated convergence theorem. Thus, to show \eqref{Goal}, it is enough to show that
\[ \lim_{n \rightarrow \infty} \mathcal{I}_n =0 ,  \]
where $\mathcal{I}_n$ is given by
\[  \mathcal{I}_n:= \left| \sum_{z \in S_n}  \overline{u^T_n}(z) \int_{B(z,\veps_n{})} \mu(T_n(x)) \psi_z(x)d \nu(x)  -    \sum_{z \in S_n}  \overline{u^T_n}(z) \int_{B(z,\veps_n)} l^T_n(x) \psi_z(x) d \nu(x)  \right|.\]

Now, for fixed $z \in S_n $, 

\begin{align}
\begin{split}
\int_{B(z , \veps_n{})} \left( \mu( T_n(x) ) - l_n( T_n(x) )   \right) \psi_z(x) d \nu(x) &= \int_{D} \left( \mu( T_n(x) ) - l_n( T_n(x) )   \right) \psi_z(x) d \nu(x)   
\\ &=   \int_{D} \left( \mu( T_n(x) ) - l_n( T_n(x) )   \right)( \psi_z(x) - \psi_z(T_n(x)) ) d \nu(x)  
\\& +  \int_{D} \left( \mu( T_n(x) ) - l_n( T_n(x) )   \right) \psi_z(T_n(x)) d \nu(x).
\end{split}
\end{align}

Observe that
\begin{align*}
\begin{split}
\left| \int_{D} \left( \mu( T_n(x) ) - l_n( T_n(x) )   \right)( \psi_z(x) - \psi_z(T_n(x)) ) d \nu(x) \right| 
\\ = \left|  \int_{B(z,  {2 }\veps_n)} \left( \mu( T_n(x) ) - l_n( T_n(x) )   \right)( \psi_z(x) - \psi_z(T_n(x)) ) d \nu(x) \right|
\\  \leq  \nu(B(z,  {2} \veps_n))\cdot \sup_{x \in B(z, {2 }\veps_n)}|\psi_z(x) - \psi_z(T_n(x))| 
\\ \leq C \nu(B(z, {2 }\veps_n)) \frac{\lVert Id - T_n \rVert_{L^\infty(\nu)}}{\veps_n},
\end{split}
\end{align*}
where the first equality comes from the fact that $\|Id - T_n \|_{L^\infty(\nu)} < \veps_n$ and the last inequality follows from \eqref{UniformBoundGradients}. The previous computations imply that 

\begin{align}
\begin{split}
 \mathcal{I}_n  & \leq   \frac{C\lVert Id - T_n \rVert_{L^\infty(\nu)}}{\veps_n} \cdot \sum_{z \in S_n } \overline{u^T_n}(z)\nu(B(z,  {2 }\veps_n)) + \sum_{z \in S_n } \overline{u^T_n}(z) {} \left \lvert \int_{D} \left( \mu( T_n(x) ) - l_n( T_n(x) )   \right) \psi_z(T_n(x)) d \nu(x) {} \right \rvert
 \\ &   \leq   \frac{C\lVert Id - T_n \rVert_{L^\infty(\nu)}}{\veps_n} \cdot \sum_{z \in S_n } \nu(B(z, {2} \veps_n))   +  \sum_{z \in S_n} \left| \frac{1}{n}\sum_{i=1}^{n}\left( \mu(\x_i) - \y_i \right) \cdot \psi_z(\x_i)   \right|
 \\& \leq   \frac{C\lVert Id - T_n \rVert_{L^\infty(\nu)}}{\veps_n} + \sum_{z \in S_n} \left| \frac{1}{n}\sum_{i=1}^{n}\left( \mu(\x_i) - \y_i \right) \cdot \psi_z(\x_i)   \right|,
 \end{split}
 \label{AuxWeak0}
 \end{align}
where in the above we have used the change of variables formula \eqref{chvar} to write
\[   \int_{D} \left( \mu( T_n(x) ) - l_n( T_n(x) )   \right) \psi_z(T_n(x)) d \nu(x) = \frac{1}{n} \sum_{i=1}^{n}\left( \mu(\x_i) - \y_i \right) \psi_z(\x_i);\]
we have also used the fact that $\overline{u^T_n}(z)$ is less than one for every $z \in S_n$.

The first term in the last line of \eqref{AuxWeak0} converges to zero as $n \rightarrow \infty$ (this follows from \eqref{NiceTranspMaps} and \eqref{vepsn}); on the other hand, Lemma \ref{LemmaAuxWeak} shows that the second term also converges to zero. Hence $\lim_{n \rightarrow \infty} \mathcal{I}_n =0$  and this finishes the proof.
\end{proof}
\section{Proof of Theorem \ref{mainTheorem2}}
\label{SecMainTheorem2}

We now move to the proof of Theorem \ref{mainTheorem2}. We impose the additional constraint:
\[  (\log(n))^{d \cdot p_d} \veps_n \ll \lambda_n \ll1.\]

%We will only prove the result for the case when the kernel $\eta$ is of the form $\eta(r)=1 $ if $r \leq 1$ and $\eta(r)=0$ otherwise.  o obtain the result for a general kernel $\eta$, we may simply follow a similar argument to the one presented in the proof of Theorem ... in \cite{GarciaTrillos2015}; essentially, by showing the result for a kernel as described before, we can easily obtain the result for a finite sum of kernels of that type, then by an approximating procedure with finite sums of that form, we may obtain the result for a general kernel $\eta$ satisfying conditions (K1) and (K2).

Let us again denote by $u_n^T$ the function $u_n^T:= u_n^* \circ T_n $, where $\left\{ T_n \right\}_{n \in \N}$ is the sequence of transportation maps from \eqref{NiceTranspMaps}. 
Up to this point, we have established that when $\veps_n$ satisfies \eqref{vepsn} and $\lambda_n$ satisfies $\veps_n \ll \lambda_n \ll 1$, 
then with probability one, the functions $u_n^*$ converge in the $TL^1$ sense towards the Bayes classifier $u_B$; 
by the very definition of $TL^1$ convergence, this is equivalent to saying that $u_n^T$ converges in the $L^1(\nu)$ sense towards $u_B$. 
Now we would like to say that the same convergence result holds for the sequence of functions $\left\{ u_n^V \right\}_{n \in \N}$,
where $u_n^V$ is the Voronoi extension (as defined in \eqref{VoronoiExtension}) of the function $u_n^*$.

Let us consider $\tilde{\veps}_n:= \veps_n - 2 \lVert T_n - Id \rVert_{L^\infty(\nu)}$. From the assumptions on $\veps_n$ and from \eqref{NiceTranspMaps}, it is clear that for large enough $n$, $\tilde{\veps}_n >0$, so without the loss of generality we assume this holds for all $n$.

Now, 
\begin{align*}
\begin{split}
 \int_{D} \lvert  u_n^T(x) - u_n^V(x) \rvert d \nu(x)   &=   \int_{D}  \left( \frac{1}{\nu( B(x,\tilde{\veps}_n))} \int_{B(x,\tilde \veps_n)} \lvert u_n^T(x) - u_n^V(x) \rvert d \nu(w)  \right)d\nu(x)    
\\&= \int_{D}\left( \frac{1}{\nu( B(x,\tilde \veps_n))}\int_{B(x,\tilde \veps_n)} \lvert u_n^T(x)- u_n^T(w) + u_n^T(w) - u_n^V(x) \rvert d \nu(w)\right) d\nu(x) 
\\&\leq  \int_{D} \left(  \frac{1}{\nu( B(x,\tilde \veps_n) )} \int_{B(x,\tilde \veps_n)} \lvert u_n^T(x)- u_n^T(w) \rvert d \nu(w) \right) d\nu(x)
\\ & +    \int_{D} \left( \frac{1}{\nu( B(x,\tilde \veps_n) )} \int_{B(x,\tilde \veps_n)} \lvert u_n^T(w)- u_n^V(x) \rvert d \nu(w) \right) d\nu(x)
\\& \leq  \frac{C}{\tilde \veps_n^d}\int_{D}\int_{B(x,\tilde \veps_n)}\lvert u_n^T(x) - u_n^T(w)\rvert d \nu(w) d \nu(x)  
\\ &+ \frac{C}{\tilde \veps_n^d}\int_{D}\int_{B(x,\tilde \veps_n)}\lvert u_n^T(w) - u_n^V(x)\rvert d \nu(w) d \nu(x)  
\\&=: C(\mathcal{I}^1_n  + \mathcal{I}^2_n ),
\end{split}
\end{align*}
where the last inequality follows from \eqref{sizeballs}. We will show now that $\int_{D} \lvert  u_n^T(x) - u_n^V(x) \rvert d \nu(x)$ converges to zero as $n \rightarrow \infty$ by showing that each of the terms $\mathcal{I}_n^1, \mathcal{I}_n^2$ converges to zero as $n \rightarrow \infty$. 
Since $u_n^T \converges{L^1(\nu)} u_B$ as $n \rightarrow \infty$, this will establish that $u_n^V \converges{L^1(\nu)} u_B$ as $n \rightarrow \infty$.

Let us first show that $\mathcal{I}^1_n \rightarrow 0$ as $n \rightarrow \infty$. Notice that for almost every $x,w \in D$ it is true that if $|T_n(x) - T_n(w)| > \veps_n $, then $|x-w|> \tilde{\veps}_n$. 
In particular, we see that for almost every $x,w \in D$
\[ \frac{1}{\tilde{\veps}_n^d}\mathds{1}_{\lvert x-w \rvert \leq \tilde{\veps}_n }  \leq \frac{1}{\tilde{\veps}_n^d} \mathds{1}_{\lvert T_n(x)-T_n(w)\rvert \leq \tilde{\veps}_n } \leq \left(\frac{\veps_n}{\tilde{\veps}_n}\right)^d \eta_{\veps_n}(T_n(x) - T_n(w)),\]
where the last inequality follows using \eqref{ExtraAssump}. Then, it follows that
\[ \mathcal{I}_n^1 \leq \left(\frac{\veps_n}{\tilde{\veps}_n}\right)^d \int_D \int_D \eta_{\veps_n}(T_n(x) - T_n(w))\lvert u_n^T(x) - u_n^T(w) \rvert d \nu(w) d \nu(x). \]
From the previous inequality and the change of variables formula \eqref{chvar}, we deduce that
\[ \mathcal{I}_n^1 \leq \frac{\veps_n}{\lambda_n} \left(\frac{\veps_n}{\tilde{\veps}_n}\right)^d  \lambda_n GTV_{n, \veps_n}(u_n^*). \]
From \eqref{Regularity}, the fact that $\frac{\veps_n}{\lambda_n}\rightarrow 0$ and $\frac{\veps_n}{\tilde{\veps}_n} \rightarrow 1$, it follows that $\mathcal{I}_n^1 \rightarrow 0$ as $n \rightarrow \infty$.

Now let us estimate the term $ \mathcal{I}^2_n$. Let us denote by $U_1^n , \dots, U_n^n$ the partition of $D$ induced by $T_n$, that is, 
\[ U_i^n := T_n^{-1}(\x_i). \]
Also, let us denote by $V_1^{n}, \dots, V_n^n$ the Voronoi partition of $D$ associated to the points $\x_1, \dots, \x_n$, that is,
\[ V_i^n := \left\{x \in D \: : \: |x- \x_i| = \min_{j=1,\dots, n} |x- \x_j|  \right\}.\]

Observe that if $x \in U_i^n$ and $w \in V_j^n$, then 
\begin{align*}
\begin{split}
 |\x_i - \x_j| & \leq |\x_i - x| + |x-w| + |w - \x_j|  
 \\& = |T_n(x) - x| + |x-w| + |w-\x_j| 
 \\& \leq |T_n(x) - x| + |x-w| + |w - T_n(w)| 
 \\& \leq |x-w| + 2 \lVert T_n - Id \rVert_{L^\infty(\nu)};
\end{split}
\end{align*}

where the second inequality follows from the fact that the closest point to $w$ among the points $\x_1, \dots, \x_n$ is $\x_j$. In particular, we see that for $x\in U_i^n$ and $w \in V_j^n$, 
\[ \frac{1}{\tilde{\veps}_n ^d} \mathds{1}_{\lvert x- w \rvert \leq \tilde{\veps}_n } \leq  \frac{1}{\tilde{\veps}_n ^d} \mathds{1}_{\lvert \x_i- \x_j \rvert \leq \veps_n } \leq \left(\frac{\veps_n}{\tilde{\veps}_n}\right)^{d}\eta_{\veps_n}(\x_i - \x_j). \]

From the previous observation, we see that

\begin{align}
\begin{split}
\mathcal{I}_n^2 & = \frac{1}{\tilde{\veps}_n^d} \sum_{i,j} \int_{U_i^n} \int_{V_j^n} \mathds{1}_{|x-w| \leq \tilde{\veps}_n} \cdot |u_n^*(\x_i) - u_n^*(\x_j) | d \nu(w) d\nu(x)      
\\ & \leq  \left(\frac{\veps_n}{\tilde{\veps}_n}\right)^{d} \sum_{i,j} \int_{U_i^n} \int_{V_j^n} \eta_{\veps_n}(\x_i- \x_j) |u_n^*(\x_i) - u_n^*(\x_j) | d \nu(w) d \nu(x)
\\& =   \left(\frac{\veps_n}{\tilde{\veps}_n}\right)^{d} \sum_{i,j} \eta_{\veps_n}(\x_i- \x_j) |u_n^*(\x_i) - u_n^*(\x_j) | \nu(V_j^n)\nu(U_i^n)
\\&= \left(\frac{\veps_n}{\tilde{\veps}_n}\right)^{d} \frac{1}{n}\sum_{i,j} \eta_{\veps_n}(\x_i- \x_j) |u_n^*(\x_i) - u_n^*(\x_j) | \nu(V_j^n)
\\& \leq \left(\frac{\veps_n}{\tilde{\veps}_n}\right)^{d}  \cdot \left(\max_{j=1,\dots,n } n \cdot \nu(V_j^n)\right)  \cdot \frac{1}{n^2}\sum_{i,j} \eta_{\veps_n}(\x_i- \x_j) |u_n^*(\x_i) - u_n^*(\x_j) | 
\\& = \left(\frac{\veps_n}{\tilde{\veps}_n}\right)^{d}  \cdot \left(\max_{j=1,\dots,n } n \cdot\nu(V_j^n)\right) \frac{\veps_n}{\lambda_n}  \lambda_n GTV_{n, \veps_n}(u_n^*),
\end{split}
\end{align}

where the third equality follows from the fact that $\nu(U_i^n)= \frac{1}{n}$ for every $i=1,\dots,n$. Now, for an arbitrary $j=1,\dots, n$, notice that if $w \in V_j^n$, then 
$|w- \x_j| \leq \lvert w- T_n(w)  \rvert \leq C\frac{ (\log(n))^{p_d}}{n^{1/d}} $ which follows from \eqref{NiceTranspMaps}. 
Thus, $V_j^n$ is contained in a ball with radius $C\frac{ (\log(n))^{p_d}}{n^{1/d}}$ and so  $\nu(V_j^n) \leq C \frac{(\log(n))^{d \cdot p_d}}{n}$ for some constant $C$ that depends on dimension and the constant $M$ from \eqref{BoundsRho}. 
Therefore,
\[ \mathcal{I}^2_n \leq  \left(\frac{\veps_n}{\tilde{\veps}_n}\right)^{d} \cdot \left( \frac{C \veps_n (\log(n))^{d \cdot p_d} }{\lambda_n} \right) \cdot\lambda_n GTV_{n , \veps_n}(u_n^*) \rightarrow 0, \quad \text{ as } n \rightarrow \infty, \]
given the assumptions on $\veps_n, \lambda_n$. This concludes the proof.

\bibliography{BiblioOver}
\bibliographystyle{siam}

\Addresses

%

%We first need the following lemma.

%\begin{lemma}
%\label{LemmaRVoronoi}
%Let $\lambda_n$ be a sequence of nonnegative numbers decaying sufficiently slow to zero. Then, with probability one the following statement holds:

%Any sequence $\left\{ u_n \right\}_{n \in \N}$ with $u_n \in L^1(\nu_n)$, is such that 
%\[   R_{\tilde{\lambda}_n}(u_n^V) \leq (1 + o(1)) R_{n,\lambda_n}(u_n). \]
%where $\tilde{\lambda}_n =$ .
%\end{lemma}
%\begin{proof}
%The only new probability estimate that we need to establish this result, is an estimate on the size of the smallest/biggest Voronoi cell induced by the points $\x_1,\dots, \x_n$. ...

%\end{proof}

%

%We know that, up to subsequence, $u_n^{*V} \rightharpoonup u$ for some function $u \in L^1(\nu)$; which follows from the boundedness of $D$ and the uniform boundedness of the sequence $\left\{ u_n^{*V} \right\}$ in the $L^\infty(\nu)$-norm.

%From Lemma \ref{LemmaRVoronoi}, we conclude that:

%\[  \sup_{n \in \N}  \frac{\tilde{\lambda}_n   }  { \veps_n} \int_{D \times D} \eta_{\veps_n}\left( x-y\right) \lvert  u_n^{*V}(x) - u_n^{*V}(y) \rvert d\nu(x) d\nu(y)   < +\infty . \]
%and thus, from Lemma \eqref{}, we deduce that

%
%\[  \int_{D} u_n(x) g_n(x)d \nu(x)   \rightarrow   \int_{D} u(x) \mu(x) d\nu(x)  \]

%where we recall $g_n$ was defined in \ref{}. From this, and the fact that the range of the functions $u_n^{*V}$ is $\left\{0,1 \right\}$, we see that,

%\[  R(u) \leq \lim_{n \rightarrow \infty} \int_{D}u_n^V d \nu(x) + \int_{D}(1- 2 g_n ) u_n^V d \nu(x)   \leq \limsup_{n \rightarrow \infty}      \]

%%%%%%%%%%%%%%%%%%%%%%%%%%%%%%%%%%%%%%%%%%%%%%%%

\end{document}